\documentclass[a4paper,11pt]{amsart}

\usepackage[top=1.65in, bottom=1.65in, right=1.3in, left=1.3in]{geometry}

\usepackage{color}
\usepackage{xcolor}

\usepackage{amssymb}
\usepackage{amsthm}
\usepackage{graphics}
\usepackage{amsmath}
\usepackage{amstext}

\usepackage{amsfonts}
\usepackage{hyperref}
\usepackage{amscd}
\usepackage[utf8]{inputenc} 
\usepackage[T1]{fontenc} 
\usepackage{fancyhdr}

\usepackage{xfrac}
\usepackage{braket}
\usepackage{mathtools}
\usepackage{stmaryrd}
\usepackage{mathrsfs}
\usepackage{extarrows}

\usepackage{microtype}

\usepackage{graphicx}
\usepackage{caption}
\usepackage{subcaption}

\usepackage{todonotes}
\usepackage{import}

\usepackage{ esint }
\usepackage{abstract}

\usepackage{tikz}
\usetikzlibrary{matrix,arrows,calc,patterns,shapes,trees,positioning}

\usepackage{enumitem}

\title{Higher bifurcations for polynomial skew products}

\begin{author}[M.~Astorg]{Matthieu Astorg}
\address{ 
Universit\'e d'Orl\'eans, Institut Denis Poisson,
 UMR CNRS 7013,
45067 Orl\'eans Cedex 2,
  France }
  \email{matthieu.astorg$@$univ-orleans.fr}
\end{author}

\begin{author}[F.~Bianchi]{Fabrizio Bianchi}
\address{ 
CNRS, Univ. Lille, UMR 8524 - Laboratoire Paul Painlev\'e, F-59000 Lille, France}
  \email{fabrizio.bianchi$@$univ-lille.fr}
\end{author}

\input pdfcolor

\newcommand{\Hh}{{\mathcal H}}

\newcommand{\Ll}{{\mathcal L}}

\newcommand{\B}{\mathbb{B}}

\def\eps{\varepsilon}

\def\lam{\lambda}

\newcommand{\mb}{\mathbb}

\newcommand{\R}{\mb R}
\newcommand{\C}{\mb C}
\newcommand{\N}{\mb N}
\newcommand{\D}{\mb D}
\newcommand{\Z}{\mb Z}

\renewcommand{\P}{\mb P}

\renewcommand{\phi}{\varphi}
\renewcommand{\epsilon}{\varepsilon}
\renewcommand{\bar}{\overline}
\renewcommand{\tilde}{\widetilde}

\newcommand{\pd}{\mathrm{Poly}(d)}

\def\({\left(}
\def\){\right)}

\newcommand{\abs}[1]{\left|#1\right|}

\newtheorem{teo*}{Theorem}

\newtheorem{teo}{Theorem}[section]
\newtheorem{defi}[teo]{Definition}
\newtheorem{cor}[teo]{Corollary}
\newtheorem{lemma}[teo]{Lemma}
\newtheorem{prop}[teo]{Proposition}
\newtheorem{remark}[teo]{Remark}

\newtheorem{claim}[teo]{Claim}

\newtheorem*{fact*}{Fact}

\DeclareMathOperator{\Lip}{Lip}

\DeclareMathOperator{\Supp}{Supp}

\DeclareMathOperator{\Bif}{Bif}

\newcommand{\res}{\mathrm{Res}}

\newcommand{\mcal}{\mathcal{M}}

\newcommand{\hcal}{\mathcal{H}}

\newcommand{\jac}{\mathrm{Jac}}

\newcommand{\tbif}{T_\mathrm{bif}}
\newcommand{\tbifz}{T_{\mathrm{bif},z}}

\newcommand{\crit}{\mathrm{Crit}}

\newcommand{\pol}{\mathrm{Poly}}

\newtheorem{lem}[teo]{Lemma}

\newcommand{\skpd}{\mathbf{Sk}(p,d)}

\DeclareMathOperator{\codim}{codim}

\DeclareMathOperator{\dist}{dist}

\begin{document}
 

\maketitle

\begin{abstract}
We continue our investigation of the parameter space of families
of polynomial skew products. 
Assuming that the base polynomial has a Julia set not totally disconnected and is neither a Chebyshev nor a power map, 
we prove that, near any bifurcation parameter, one can find
parameters where $k$ critical points bifurcate 
 \emph{independently}, with $k$ up to the dimension of the parameter space.
This is a striking difference with respect to
the one-dimensional case. 
 The proof is based on  a variant
of the inclination lemma, applied to the postcritical set at a
Misiurewicz parameter.
By means of an analytical criterion for the non-vanishing
of the self-intersections of the bifurcation current,
we deduce
the equality of the supports of the bifurcation current and
the bifurcation measure for  such families.
Combined with results by Dujardin and Taflin, this also implies
that the support of the bifurcation measure in these families has
non-empty interior.

As part of our proof we construct, in these families, subfamilies
of codimension 1 where the bifurcation locus has non empty interior.
This provides a new independent proof of the existence
of holomorphic families of arbitrarily large dimension whose bifurcation locus
has non empty interior. Finally, it shows that the Hausdorff dimension
of the support of the bifurcation measure is maximal at any point of its support.
\end{abstract}

\section{Introduction}

Polynomial skew products are regular polynomial endomorphisms of $\C^2$
  of the form
$f(z,w)= (p(z), q(z,w))$, for $p$ and $q$ polynomials of a given degree $d\geq 2$.
\emph{Regular} here means that the coefficient of $w^d$ in $q$ is non zero, which is equivalent to the
 extendibility
of these maps as holomorphic self-maps of $\P^2$.
Despite their specific forms, these maps already provided examples of new phenomena with respect to the 
established theory of
one-variable polynomials or rational maps, see for instance
\cite{astorg2014two,dujardin2016nonlaminar,taflin_blender}.
We started in \cite{astorg2018bifurcations} a detailed
study of the parameter space of such maps.

We will denote in what follows by $\skpd$ the family of all polynomial skew products
of a given degree $d$ over a fixed base polynomial $p$
up to affine conjugacy,
and denote by $D_d$ its dimension. 
Following \cite{bbd2015}
it is possible to divide the parameter space of the family
$\skpd$ (identified with $\C^{D_d}$)
into two dynamically defined subsets: the \emph{stability locus} and the \emph{bifurcation locus}.
The bifurcation locus coincides with the support of $dd^c L_v $, where $L_v (f)$ denotes the vertical
Lyapunov function of $f$, see
\cite{jonsson1999dynamics,astorg2018bifurcations}. We gave
in \cite{astorg2018bifurcations} a description of the bifurcation locus and current
in terms 
of natural bifurcation loci and currents associated to
the vertical fibres,  and a classification of unbounded hyperbolic components in the quadratic case.

For families of rational maps, the study of the self-intersections 
of the bifurcation current
(which are meaningful because of the continuity of its potential)
 was started in
 \cite{bassanelli2007bifurcation}, see also \cite{pham,dujardin2008distribution,dujardin2011bifurcation}.
A geometric interpretation of the support of these currents is the following:
the support of $\tbif^k:= \tbif^{\wedge k}$ is the locus where $k$  critical points bifurcate
\emph{independently}. 
Moreover,
 the 
current $\tbif^k$
 is known to equidistribute 
many kinds of dynamically defined parameters, such as maps possessing 
$k$ cycles of prescribed multipliers and 
periods tending to infinity (see, e.g., \cite{bassanelli2007bifurcation,gauthier2016equidistribution}). 
This gives rise to a natural stratification of the bifurcation locus as
\[
\Supp \tbif \supseteq \Supp \tbif^2f \supseteq \dots  \supseteq \Supp \tbif^{k_\mathrm{max}}
\]
where $k_{\mathrm{max}}$ is the dimension of the parameter space.
The inclusions above are not equalities in general, and are for instance strict when considering the family of
all polynomial or rational maps of a given degree (where $k_{\mathrm{max}}$ is equal to $d-1$  and $2d-2$, respectively). It
is worth pointing out that this stratification 
is often
compared with
an analogous stratification for the Julia sets of endomorphisms of $\P^k$ (given by the supports
of the self-intersections of the Green current, see for instance \cite{dinh2010dynamics}). 
We refer to \cite{dujardin2011bifurcation}  for a more detailed 
exposition.

In \cite{astorg2018bifurcations}, the authors have proved the first equidistribution property for the bifurcation current 
$\tbif$ in families of endomorphisms of projective spaces in any dimension, 
including polynomial skew products: for a generic $\eta \in \C$, 
the bifurcation current $\tbif$  equidistributes the polynomial
skew products with a cycle of period tending to infinity and vertical multiplier $\eta$. 
The arguments could easily be adapted to prove a similar statement for the 
bifurcation
currents $\tbif^k$:
given generic $\eta_1, \ldots,
\eta_k \in \C$, 
$k \leq k_{\mathrm{max}}$,
the bifurcation 
current $\tbif^k$
 equidistributes skew products having
 $k$
  cycles of periods tending to infinity and respective
vertical multipliers  $\eta_1, \ldots, 
\eta_k\in \C$.
It is then natural to ask whether the supports of the bifurcation currents 
still
give a natural stratification of the bifurcation locus.
  
   The goal of this paper is to 
 show that the situation in families of higher dimensional dynamical systems
  is completely different from the one-dimensional counterpart.
   Namely, we establish the following result.

\begin{teo}\label{teo_main}
Let $p$ be a polynomial with Julia set not totally disconnected,
	 which is neither conjugated to $z \mapsto z^d$
	nor to a Chebyshev polynomial.
Let $\skpd$ denote the family of polynomial skew products of degree $d\geq 2$ over the base polynomial $p$, up to affine conjugacy, 
and
let $D_d$ be its dimension.
Then
 the associated  bifurcation current $\tbif$ satisfies 
\[
\Supp \tbif \equiv \Supp \tbif^2 \equiv\dots \equiv \Supp \tbif^{D_d}.
\]
\end{teo}

Theorem \ref{teo_main} is stated for the full family $\skpd$ of all 
polynomial 
skew products of degree $d$ over $p$ (up to affine conjugacy). 
One could ask whether such a result holds for algebraic subfamilies of $\skpd$: clearly, some special subfamilies
have to be ruled out, such as the family of trivial product maps of the form $(p,q) : (z,w) \mapsto (p(z), q(w))$. 
A less obvious example in degree 3 is given by the subfamily of polynomial
skew products over the base polynomial
$z \mapsto z^3$ of the form
$$f_{a,b} : (z,w) \mapsto (z^3, w^3 + awz^2+bz^3), \quad \quad (a,b) \in \C^2.$$
One can check that $f_{a,b}$ is semi-conjugated to the product
map $(z,w) \mapsto (z^3, w^3 + aw+b)$, and therefore that 
$\Supp \tbif^2(\Lambda) \subsetneq \Supp \tbif(\Lambda)$, where  $\Lambda:=\{f_{a,b}, (a,b) \in \C^2 \}$.

The proof of Theorem \ref{teo_main} indeed uses the fact that the family $\skpd$ is general enough so that 
it is possible to perturb a bifurcation 
parameter to change the 
dynamical behaviour of a critical point in a vertical fibre without affecting all other fibres. 
It would be  interesting
to classify algebraic subfamilies of $\skpd$ 
that, like $\Lambda$, are degenerate in the sense that a bifurcation in one fibre implies a bifurcation in all other fibre; 
for such families, the conclusion of Theorem \ref{teo_main} will not hold. 
 Likewise, it would be natural to try to extend
Theorem \ref{teo_main} to other families with a similar fibred structure,
 see for instance \cite{dupont2018dynamics}. 
To do this, one 
should first ensure that such a family is large enough in the sense above.

\medskip

The proof of Theorem \ref{teo_main} essentially consists of two ingredients, respectively
of analytical and geometrical flavours.

The first is 
 an analytical sufficient
condition for a parameter to be in the support of $\tbif^k$.
This is
 inspired by analogous results by Buff-Epstein 
\cite{buff2009bifurcation} and 
Gauthier \cite{gauthier2012strong}
in the context of rational
maps, and is based on the notion of \emph{large scale condition} introduced in \cite{astorg2019collet}.
It is a way to give a quantified meaning to
the
 \emph{simultaneous independent bifurcation of multiple critical points}, and to exploit this condition
 to prove the
 non-vanishing 
of $\tbif^k$.
This part does not require
essentially new arguments and is presented in Section \ref{s:beg}.

The second ingredient is a procedure
to build these multiple independent bifurcations at a common parameter
starting from a simple one. The idea is to start with
 a parameter with a \emph{Misiurewicz} bifurcation, i.e., a non-persistent collision between
 a critical orbit and a repelling point, and to construct a new parameter nearby where
  two 
 -- and actually, 
 $D_d$ --
independent  Misiurewicz bifurcations occur simultaneously. 
Here we say that 
$k$ Misiurewicz relations are 
\emph{independent} at a parameter $\lam$
if the intersection of the $k$ hypersurfaces given by the Misiurewicz relations has codimension
$k$ in $\skpd$, see Subsection \ref{ss-misiurewicz-relations}, and we denote by $\Bif^k$ the closure of such parameters.

This geometrical
 construction is our main technical result, and the main contribution of this paper. Together with
 the analytic arguments mentioned above
 (which give $\Bif^k \subseteq \Supp \tbif^k$ for all $1\leq k\leq D_d$)
 and the trivial inclusion $\Supp \tbif^{D_d}\subseteq \Supp \tbif$,
  it implies Theorem \ref{teo_main}.

\begin{teo}\label{teo_main_no_current}
Let $p$ be a polynomial with Julia set not totally disconnected,
	 which is neither conjugated to $z \mapsto z^d$
	nor to a Chebyshev polynomial.
Let $\skpd$ denote the family of polynomial skew products of degree $d\geq 2$ over the base polynomial $p$, up to affine conjugacy, 
and
let $D_d$ be its dimension. Then
\[
\Bif = \Bif^2 = \dots  = \Bif^{D_d}.
\]
\end{teo}

In order to construct the desired Misiurewicz parameter, we will consider the motion 
of a sufficiently large hyperbolic subset of the Julia set near a parameter in the bifurcation locus.
This hyperbolic set needs to satisfy some precise properties, and this is where the
assumptions on $p$ come into play. The construction, presented in Section \ref{s:vlifs}, 
uses tools from the thermodynamic formalism
of rational maps, and more  generally of endomorphisms of $\P^k$, as explained in Appendix \ref{as:proof-lemma-fibered}.
Once the hyperbolic set is constructed, the proof proceed by induction. We show that, given a
Misiurewicz
relations  satisfying a given list of further properties (see Definition \ref{def:goodmis}), it is possible
to construct, one by one, the extra Misiurewicz relations happening simultaneously. The 
general construction and the application in our setting are 
given in Sections \ref{s:machine} and \ref{s:proof_teo}, respectively.

\medskip

Our main theorems
and the method developed for their proof 
have a number of consequences and corollaries. 
We list here a few of them.

\begin{cor}\label{big_families_mubif_all}
Let $p$ be a polynomial with Julia set not totally disconnected,
	 which is neither conjugated to $z \mapsto z^d$
	nor to a Chebyshev polynomial.
 Near any bifurcation parameter in $\skpd$ there exist algebraic subfamilies
 $M^k$ of $\skpd$ of any dimension $k<D_d$ such that the support of the bifurcation measure
of $M^k$ has non-empty interior in $M^k$.
\end{cor}

These families are 
 given by the maps satisfying a given critical relation.
 Notice that $d$ (and thus $D_d$) can be taken arbitrarily large.
  This result is for instance an improvement
 of the main result in \cite{bt_desboves}, where 1-parameter families with  the same property
  are constructed.

\medskip

More strikingly, in \cite{dujardin2016non,taflin_blender}, Dujardin
and Taflin construct open sets in the bifurcation locus in the family $\Hh_d (\P^k)$
of all endomorphisms of $\P^k$, $k\geq 1$, 
of a given degree $d\geq 2$
(see also \cite{biebler2019lattes} for further examples).
Their strategy also works when considering the subfamily of polynomial skew products
(and actually these open sets are built close to this family). Combining
Theorem \ref{teo_main} with their result we thus get the following
consequence.

\begin{cor}\label{cor_open}
Let $p$ be a polynomial with Julia set not totally disconnected,
	 which is neither conjugated to $z \mapsto z^d$
	nor to a Chebyshev polynomial.
 The support of the bifurcation  measure in $\skpd$ 
has non empty interior.
\end{cor}

Notice that it is not known whether the bifurcation locus is the closure of its 
interior
(see the last paragraph in \cite{dujardin2016non}). Hence, a priori, the open sets
as above could exist only in some regions of the parameter space.
 The last consequence of our main theorems
  is 
 a uniform and optimal bound
 for the Hausdorff dimension of the support of the bifurcation measure, which is
a generalization to this setting of the main result in \cite{gauthier2012strong}.

\begin{cor}\label{cor_intro_hausdorff}
Let $p$ be a polynomial with Julia set not totally disconnected,
	 which is neither conjugated to $z \mapsto z^d$
	nor to a Chebyshev polynomial.
The Hausdorff dimension 
of the support of the bifurcation measure
in $\skpd$
 is maximal at all points of its
 support.
\end{cor}

Notice that, in the family of all endomorphisms of a given degree, such a uniform estimate 
is not known even for the bifurcation locus, see \cite{berteloot2018hausdorff} for some local
estimates.

\subsection*{Acknowledgements}
This project has received funding from
 the French government through the programs 
 I-SITE ULNE / ANR-16-IDEX-0004 ULNE,  
LabEx CEMPI /ANR-11-LABX-0007-01, and
ANR JCJC Fatou ANR-17- CE40-0002-01, from
the CNRS  through the
 program PEPS JCJC 2019, and
 from
 the Louis D. Foundation
 through the project "Jeunes Géomètres".

This work was partly done in Kyoto
during the first author's visit to KIT through a
FY2019 JSPS (Short Term)
Invitational Fellowship
for Research in Japan and
 the authors' visit to RIMS.
 The authors would like to thank 
 these Institutions and Y\^usuke Okuyama
  for their support, warm welcome, and excellent work conditions.

\section{Notations and preliminary results}\label{s:prelim}

\subsection{Notations}\label{ss:notations}
We collect here the main notations that we will use
through all the paper. We refer to \cite{astorg2018bifurcations} and \cite{jonsson1999dynamics} for more 
details.

Given a polynomial skew product of degree $d\geq 2$ of the form $f(z,w)=(p(z),q(z,w))=: (p(z),q_z (w))$, we
will write the $n$-th iterate of $f$
as
\[
f^{n} (z,w) = (p^{n} (z), q_{p^{n-1} (z)}\circ \dots \circ q_z (w))=: (p^n (z), Q^n_z (w)).
\]
In particular, if $z_0$ is a $n_0$-periodic point for $p$, the map $Q^{n_0}_{z_0}$ is the return map
to the vertical fibre $\{z_0\}\times \C$ and is a polynomial of degree $d^{n_0}$. For every $z$ in the Julia set
$J_p$ of $p$, we denote by $K_z\subset \C$ the set of of points $w$
such that the sequence
 $\{Q_z^n\}$ is bounded and by $J_z$ the boundary of $K_z$. 
 Given a subset $E\subseteq J_p$, we denote  $J_E:= \bar{ \cup_{z\in E} \{z\}\times J_z}$.

Let us now denote by $(f_\lam)_{\lam \in M}$ a holomorphic family of polynomial
skew products of a given degree $d \geq 2$,
that is a holomorphic map
$F\colon M\times \C^2 \to \C^2$
such that $f_\lam := F(\lam, \cdot)$ is a polynomial
skew product of degree $d$ for all $\lam \in M$.
 We will denote 
  by $\skpd$ the
  family of all
polynomial skew products of degree $d$
with the given polynomial $p$ as first component, 
 up to affine conjugacy. 
We set $D_d  := \dim \skpd$.
An explicit description of these families
in the case $d=2$ is given in \cite[Lemma 2.9]{astorg2018bifurcations}, 
the general case 
is similar.

 \begin{lemma}\label{lemma_param_skpd}
Every polynomial skew product of degree $d\geq 2$ over a polynomial $p$ 
is affinely conjugated to a map of the form
\[
(z,w) \mapsto \big( p(z), w^d + \sum_{j=0}^{d-2} w^{j} A_j(z) \big) \quad \mbox{ with } \quad  \deg_z A_j = d-j.
\]
\end{lemma}

We are interested in \emph{bifurcations}  within families of polynomial skew products. Following \cite{bbd2015},
the \emph{bifurcation locus} $\Bif$ is defined as the support of the $(1,1)-$positive closed
current $\tbif := dd^c_\lam L(\lam)$ on $M$, where 
$L(\lam)$ is the Lyapunov function associated to $f_\lam$ with respect to its measure of maximal entropy.
In the case of polynomial skew
products, the function $L$ has a quite explicit description.
Indeed, by \cite{jonsson1999dynamics} we have $L(\lam) =L_p (\lam) + L_v (\lam)$, where
\begin{equation}\label{eq_lyapunov}
L_p (\lam) = \log d +\sum_{z \in C_{p_\lam}} G_{p_\lam} (z) \mbox{ and }
L_v (\lam) = \log d +\int \Big( \sum_{w  : q'_{\lam,z}(w)=0}G_{\lam} (z,w)\Big) \mu_{p_\lam} (z).
\end{equation}
Here $\mu_{p_\lam}, G_{p_{\lam}}, C_{p_\lam}$ are
 the measure of maximal entropy, the Green function and the critical set
 (whose points are counted with multiplicity) of $f_\lam$ and $p_\lam$ respectively,
  and
$G_{\lam}(z,w):= \lim_{n\to \infty} \frac{1}{n}\log^+ \big\|Q^{n}_{\lam,z} (w)\big\|$
is the \emph{non-autonomous Green function} for the family $\{Q^n_{\lam,z}\}_{n\in \N}$.
The current $T_p := dd^c_\lam L_p(\lam)$ is positive and closed. We proved in \cite[Proposition 3.1]{astorg2018bifurcations}
that
$T_v :=dd_\lam^c L_v = \tbif - T_p$ is also positive and closed. This allowed us to define the
 \emph{vertical bifurcation} in any family of polynomial skew products. This was generalized in \cite{dupont2018dynamics}
 to cover families of endomorphisms of $\P^k (\C)$ preserving a fibration. Of course, when $p$ is constant
 we have $\tbif= T_v$.

\subsection{Families defined by Misiurewicz relations}\label{ss-misiurewicz-relations}
By \cite{bbd2015, b_misiurewicz} the bifurcation locus 
of a family $(f_\lam)_{\lam \in M}$
coincides with the closure of the set of \emph{Misiurewicz parameters}, i.e., parameters for which we have a
non-persistent intersection between some component of the post critical  set and the motion
of some repelling point. 
More precisely, in our setting
take $\lam_0 \in \skpd$ and let
 $M$ be any holomorphic
 subfamily of $\skpd$ such that $\lam_0 \in M$.
A \emph{Misiurewicz relation} 
for $f_{\lam_0}$
  is an equation of the form
   $f_{\lam_0}^{n_0}(z_0,c_0)=(z_1,w_1)$
where $(z_1,w_1)$ is a repelling periodic point of period $m$ for $f_{\lam_0}$, 
and $q_{z_0,\lam_0}'(c_0)=0$.

Assume 
that
 $c_0$ is a simple root of $q_{z_0,\lam_0}'$ (this assumption
could be 
removed, but we  keep it here
for the sake of simplicity). 
Then there is a unique holomorphic map $\lam \mapsto c(\lam)$ 
 defined on a neighbourhood of $\lam_0$ in $\skpd$
 such that $c(\lam_0)=c_0$. 
 Similarly, it is possible to locally follow holomorphically the repelling point 
 $(z_1,w_1)$ as $\lam \mapsto (z_1, w_1(\lam))$.

The Misiurewicz relation $f_{\lam_0}^{n_0}(z_0,c_0)=(z_1,w_1)$ is said to be \emph{locally persistent in $M$} if 
$f_\lam^{n_0}(z_0,c(\lam))=(z_1,w_1(\lam))$ for all $\lam$ in a neighbourhood of $\lam_0$ in $M$.
If this is not the case,
the equation 
$f_\lam^{n_0}(z_0,c(\lam))=(z_1,w_1(\lam))$ defines a germ of analytic
 hypersurface in $M$
at $\lam_0$, which is
open inside  the algebraic hypersurface of $M$ given by
$\{\lam \in M: \res_w(q_{\lam,z_0}', Q_{\lam,z_0}^{n_0+m}-Q_{\lam,z_0}^m)=0 \}$. 
Here, $\res_w(P,Q)$ denotes the resultant of two polynomials $P,Q \in A[w]$, where $A:=\C[\lam]$; it is therefore an element of $A$.
 Notice that 
this algebraic hypersurface consists of all $\lam \in M$ 
such that some critical point in the fibre at $z_0$ lands after $n_0$ iterations on some
 periodic point of period dividing $m$. 
 We also
say in this case that $\lam_0$ is a \emph{Misiurewicz parameter in $M$}. 
If the Misiurewicz relation is non-persistent in $M$, we denote by
$M_{(z_0,c), (z_1,w_1),n_0}$ 
(or by
$M_{(z_0,c_0), (z_1,w_1 (\lam_0)),n_0}$ if we wish to emphasize the starting parameter $\lam_0$)  this irreducible component and 
we call it the locus where the relation is locally preserved.
We may avoid 
mentioning the periodic point if this does not create confusion.

\subsection{The unicritical subfamily $U_d \subset \skpd$}\label{ss_unicritical}

We
	consider here the
	 \emph{unicritical
	subfamily} $U_d\subset \skpd$ 
	given by
 \begin{equation}\label{eq_unicritical}
  U_d:= \{f(z,w)=(p(z), w^d+a(z))\},
 \quad  a(z) \in \C_d[z] \sim \C^{d+1}.
 \end{equation} 
Thus,   $U_d$ has dimension $d+1$. We
 parametrize
it with $\lam := (a_0, \dots, a_d)$, where 
the $a_i$ are the coefficients of $a(z)$.
We will write $\lam(z_0)=0$ when $z_0$ is a root
of the polynomial $a(z)$ associated to $\lam$, and similarly
$\lam' (z_0)=0$ when $z_0$
is a root of $a'(z)$.

 We 
can compactify
this parameter space
 to $\P^{d+1}$ and we denote by $\P^{d}_\infty$
the hyperplane at infinity.
Notice that, 
unless $p'(z_0)= 0$,
 $(z_0,0)$ is the only critical point for $f_\lam$
in the fibre $\{z=z_0\}$
 (this justifies the name chosen for this family, coherently with the name of the
one dimensional unicritical family $f_{\lam}(z)= z^d + \lam$).

\begin{lemma}\label{lemma_claim_12}
There exist two positive constants $C_1, C_2$ 
such that, for all $\lam\in U_d$
 and for all $z\in J_p$, we have 
$K_z(f_\lam) \subset \D(0,C_1 + C_2 |\lam|^{1/d})$.
Moreover, if $\lam_j\in M$
is a sequence with $|\lam_j|\to \infty$
and $[\lam_j]\to [\lam_\infty]$ for some $\lam_\infty$ such that $\lam_\infty (z_0)\neq 0$,
 then for all $w_j \in K_{z_0} (f_{\lam_j})$ we
have $|w_j| \asymp |\lam_j|^{1/d}$ as $j\to \infty$.
\end{lemma}

\begin{proof}
Set $A(\lam) := \max_{z \in J_p} |a(z)|$. Observe that we have $A(\lam) =\mathcal{O} ( |\lam|)$ as $|\lam|\to \infty$, hence
there exists $C_0>2$ such that $A(\lam)\leq C_0 |\lam|$ for all $\lam \in U_d$.
It follows that, if $w$ satisfies $|w|>C_0 |\lam|^{1/d}$, then
for any $z\in J_p$
we have
 $|q_z(w)| > C_0^2 |\lam| - A(\lam)\geq ( C_0^2  - C_0) |\lam| > C_0 |\lam|$.
 This proves the first assertion.
 
 For the second assertion, let 
 $a^{(j)} (z)$
 be the polynomial
associated to $\lam_j$.
 Take $w_j\in K_{z_0} (f_{\lam_j})$.
Hence, $q_{\lam_j,z_0}(w_j) \in K_{p(z_0)} (f_{\lam_j})$. 
 By the first part of the statement, we have 
 $|w_j^d + a^{(j)}(z_0)| = |q_{\lam_j,z_0}(w)|\leq C_1 + C_2 |\lam_j|^{1/d}$. 
 Since $\lam(z_0)\neq 0$, we have $|a^{(j)}(z_0)|\asymp |\lam_j|$ as $\lam_j\to \infty$. 
 Hence, $|w_j^d|\asymp |\lam_j|$,
  which gives $|w_j|\asymp |\lam_j|^{1/d}$.
\end{proof}

Let us now consider the intersection of a Misiurewicz hypersurface in $\skpd$
with $U_d$. This (when not empty)
is a Misiurewicz hypersurface in $U_d$. Since the only critical points
for maps in $U_d$  that can give non-persistent Misiurewicz relations in $U_d$
are of the form $(z_0,0)$ (and these all have multiplicity $d$),
we see that any Misiurewicz hypersurface of $U_d$ has the form
\begin{equation}\label{eq_mis_ud}
Q^{n}_{\lam,z_0} (0) = Q^{n+m}_{\lam,z_0} (0) \mbox{ for some } m,n\geq 1
\mbox{ and } z_0 \in J_p \mbox{ with } p^{n+m}(z_0)=p^{n}(z_0).
\end{equation}
For simplicity, we denote by $M_{z_0,n,m}$ the hypersurface defined by
\eqref{eq_mis_ud}.

\begin{lemma}\label{lemma_ez}
For any non-empty Misiurewicz hypersurface $M_{z_0,n,m}\subset U_d$ of the form \eqref{eq_mis_ud},
 the accumulation
on $\P^d_\infty$ 
of $M_{z_0, n,m}$
is precisely given by $E_{z_0} := \{ [\lam] \colon 
\lam(z_0) =0\}$.
\end{lemma}

\begin{proof}
It follows from Lemma \ref{lemma_claim_12} that the accumulation of $M_{z_0,n,m}$
on $\P^d_\infty$
is included in $E_{z_0}$. On the other hand, the restriction of $M_{z_0,n,m}$
to any $2$-dimensional subfamily of $U_d$ cannot be compact.
By considering, for every point in $E_{z_0}$, an
 affine $2$-dimensional
 subfamily whose line of intersection with $\P^d_\infty$
meets $E_{z_0}$ only in the given point,  we see
that the inclusion is actually an equality.
\end{proof}

\begin{lemma}\label{lemma_dir_inv}
For any non-empty
Misiurewicz hypersurface $M_{z_0,n,m}\subset U_d$ of the form \eqref{eq_mis_ud},
the
 non-vertical eigenspace  
of $(df_{\lam}^m)_{(z_1, w_1(\lam))}$ 
at $(z_1, w_1(\lam)):=  (p^{n} (z_0), Q^{n}_{\lam, z_0} (0))$
is generated by the vector
\[v_{\lam}:=
\left(
1, 
\frac{ 
 \frac{
\partial Q_{\lam,z}^m(w)}{\partial z}\big|_{(z,w)=(z_1, w_1(\lam))} }{ \frac{\partial Q_{\lam,z}^m (w)}{\partial w}\big|_{(z,w)= (z_1,w_1(\lam))} - (p^m)'(z_1)}
\right).\]
In particular, 
if
$z_0 \notin \{p^{n}(z_0), \dots, p^{n+m-1} (z_0)\}$,
given 
$\lam_\infty$ such that
$\lam'_\infty (p^i (z_0) )\neq 0$
for all $n\leq i <n+m$
and
a sequence $\lam_j \in M_{z_0, n,m}$ with
 $|\lam_j|\to \infty$ and $[\lam_j] \to [\lam_\infty]$,
  the second component
$v_{\lam_j}^{(2)}$  
   of $v_{\lam_j}$ as above
 satisfies
 \begin{equation}\label{eq_dir_inv}
|v^{(2)}_{\lam_j}| = \mathcal{O} (|\lam_j|^{1/d})
\mbox{ as } j\to \infty.
 \end{equation}
\end{lemma}

\begin{proof}
We have
\[
(df^m_\lam)_{z_1,w_1(\lam)} = 
\left( \begin{array}{cc}
(p^m)' (z_0)& 0\\
\frac{\partial Q_{\lam,z}^m(w)}{\partial z}\big|_{(z,w)=(z_1,w_1(\lam))}  & \frac{\partial Q_{\lam,z}^m(w)}{\partial w}\big|_{(z,w)=(z_1,w_1(\lam))}
\end{array}\right),
\]
from which we deduce the first assertion.
A direct computation shows that, for every $\lam \in M_{z_0,n,m}$,
\begin{equation}\label{eq_dz}
\begin{aligned}
\frac{\partial Q_{\lam,z}^m(w)}{\partial z}\Big|_{(z,w)=(z_1,w_1(\lam))}
&  =
\sum_{i=0}^{m-1}
a'(p^i (z))
\prod_{\ell=i+1}^{m-1} q'_{p^\ell (z_1)}  (Q_{\lam,z_1}^\ell (w_1(\lam)))\\
&=
\sum_{i=0}^{m-1}
a'(p^i (z_1))
\prod_{\ell=i+1}^{m-1} d \left[ Q_{\lam,z_1}^\ell (w_1(\lam))\right]^{d-1}
\end{aligned}
\end{equation}
and
\[
\frac{\partial Q_{\lam,z}^m(w)}{\partial w}\Big|_{(z,w)=(z_1,w_1(\lam))}  =
\prod_{\ell=0}^{m-1} q'_{p^\ell (z_1)}  (Q_{\lam,z_1}^j (w_1(\lam)))
=
\prod_{\ell=0}^{m-1} d \left[Q_{\lam,z_1}^\ell (w_1(\lam))\right]^{d-1}
\]
where $a$  is the polynomial associated to $\lam$.

Let us now consider the evaluations of the expressions above at a
sequence $\lam_j$ as in
the statement,
and
let us denote 
by $a^{(j)}$ the polynomial associated to $\lam_j$.
Since
$\lam'_\infty (p^i (z_0) )\neq 0$ for all 
$n\leq i < n+m$,
 we have $| (a^{(j)})' (p^i (z_0))|\asymp |\lam_j|$ as $j\to \infty$.
Moreover, 
by Lemma \ref{lemma_claim_12}, for all $0\leq \ell < m$ we have $|Q^\ell_{z_1,\lam_j} (w_1(\lam_j))|\asymp |\lam_j|^{1/d}$.
Hence,
both the expressions
above diverge as $|\lam_j|\to \infty$ with $[\lam_j] \to [\lam_\infty]$,
and 
the largest term  in the sum in the  last term of  
\eqref{eq_dz} is that corresponding to $i=0$.
 Hence, as $j\to \infty$,
we have
\[
\begin{aligned}
|v^{(2)}_{\lam_j}|
& \asymp 
\frac{
\abs{ (a^{(j)})'(z_1)\prod_{\ell=1}^{m-1} d (Q_{\lam_j,z_1}^\ell (w_1(\lam_j)))^{d-1} }
}{
\abs{\prod_{\ell=0}^{m-1} d (Q_{\lam_j,z_1}^\ell (w_1(\lam_j)))^{d-1}}} 
= \frac{|(a^{(j)})'(z_1)|}{ d  (w_1(\lam_j) )^{d-1} }\\
& = \mathcal{O}
\left( \frac{|\lam_j|}{|\lam_j|^{(d-1)/d}} \right) = \mathcal{O} (|\lam_j|^{1/d}),
\end{aligned}\]
where in the last steps we used again Lemma \ref{lemma_claim_12}.
\end{proof}

\begin{lemma}\label{lemma_dir_pc}
For any non-empty Misiurewicz hypersurface $M_{z_0,n,m}\subset U_d$ of the form \eqref{eq_mis_ud},
 the image of $(df_\lam^n)_{(z_0,0)}$
is generated by the vector
\[u_{\lam}: = \left( 1, 
\frac{ 
\frac{\partial Q_{\lam,z}^n(w)}{\partial z}\big|_{(z,w)=(z_0,0)} 
}{(p^n)'(z_0)}
 \right).
 \]
 In particular, 
 given 
 $\lam_\infty$ such
  $\lam'_\infty (p^i(z_0))\neq 0$ for $0\leq i \leq n-1$
  and
 a sequence $\lam_j \in M_{z_0, n,m}$
 with
 $|\lam_j|\to \infty$ and
  $[\lam_j]\to [\lam_\infty]$,
  the second component 
  $u^{(2)}_{\lam_j}$
  of $u_{\lam_j}$ as above
 satisfies
 \begin{equation}\label{eq_dir_pc}
|u^{(2)}_{\lam_j}| 
\asymp |\lam_j|^{\frac{n(d-1)+1}{d}}
\mbox{ as } j\to \infty.
 \end{equation}
\end{lemma}

\begin{proof}
Since
\[
(df^n_\lam)_{z_0,0} = 
\left( \begin{array}{cc}
(p^n)' (z_0)& 0\\
\frac{\partial Q_{\lam,z}^n(w)}{\partial z}\big|_{(z,w)=(z_0,0)}  & 0
\end{array}\right),
\]
the first part of the statement is immediate.
A computation as in Lemma \ref{lemma_dir_pc} 
gives
\[
\frac{\partial Q_{\lam,z}^n(w)}{\partial z}\Big|_{(z,w)=(z_0,0)} 
=
\sum_{i=0}^{n-1}
a'(p^i (z_0))
\prod_{\ell=i+1}^{n-1} d (Q^\ell_{\lam,z_0} (0))^{d-1}
\]
(where again $a$ is the polynomial associated to $\lambda$)
and, by Lemma \ref{lemma_claim_12}, the above expression diverges as $j\to \infty$
when evaluated at $\lam_j$ as in the statement. Moreover, as $j\to \infty$,
denoting by $a^{(j)}$ the polynomial
associated to $\lam_j$,
 we have
\[\begin{aligned}
|u^{2}_{\lam_j}|\asymp
\left|
\frac{\partial Q_{\lam_j,z}^n(w)}{\partial z}\Big|_{(z,w)=(z_0,0)} \right|
& \asymp
\Big| (a^{(j)})'(z_0)
\prod_{\ell=1}^{n-1} d (Q_{\lam_j,z_0}^\ell (0))^{d-1}\Big|
\\
& \asymp |\lam_j|^{1+ \frac{(n-1)(d-1)}{d}}= |\lam_j|^{\frac{n(d-1)+1}{d}},
\end{aligned}
\]
where we used the facts that $|(a^{(j)})'(z_0)|\neq 0$ for sufficiently large $j$, and hence $|(a^{(j)})'(z_0)|\asymp |\lam_j|$,
and that $Q^\ell_{\lam_j,z_0}(0) \in K(f_{\lam_j})$, and hence $|Q^\ell_{\lam_j,z_0}(0)|\asymp |\lam_j|^{1/d}$ by Lemma \ref{lemma_claim_12}.
\end{proof}

\subsection{Higher bifurcations currents and loci}\label{s:decomposition}

Higher bifurcation currents
for families of polynomials (or rational maps)
in one variable 
were introduced in \cite{bassanelli2007bifurcation},
 see also \cite{dujardin2008distribution},
with the aim of understanding the 
loci where simultaneous and independent bifurcations happen, from an analytical
point of view. Since the Lyapunov function is continuous with respect to
the parameters 
\cite{dinh2010dynamics}, it is indeed
meaningful
 to
consider the self-intersections $\tbif^k:= \tbif^{\wedge k}$
of the bifurcation current, for every $k$ up to the dimension of the parameter space. 
The measure obtained by taking
the maximal power is usually referred to as the \emph{bifurcation measure}.

While in dimension one it is quite natural
to associate a  geometric meaning to
$\Supp (\tbif^k)$ (as, for instance, the points where $k$ independent Misiurewicz
relations happens, in a quite precise sense, see, e.g., \cite{dujardin2011bifurcation}), 
in higher dimensions 
the critical set is of positive dimension 
and thus this interpretation is far less clear.

The following result gives a first step in the interpretation of the higher bifurcations
as average of non-autonomous counterparts of the
classical one-dimensional objects,
valid in any family of polynomial skew products
over a fixed base $p$. 
An interpretation of the non-autonomous factors will 
be the object
 of 
 Section \ref{s:higher_current}.
The case of general polynomial skew products is completely analogous, 
and the following should be read as a decomposition for the vertical bifurcation $T_v^k = (dd^c L_v)^k$, 
see Section \ref{ss:notations}.
 Given $\underline z:=(z_1, \dots, z_k) \in J_p^k$, we denote by 
$T_{\underline z}$ the
current
$T_{\underline z}= {\tbif}_{z_1} \wedge \dots \wedge {\tbif}_{z_k}$, where
for every $z\in J_p$ we set
${\tbif}_{z} := dd^c_\lam  \Big(\sum_{w  : q'_{\lam,z}(w)=0}G_{\lam} (z,w)\Big)$, see
\cite[\S 2.4]{astorg2018bifurcations}.

\begin{prop}\label{prop_decomposition}
Let $(f_\lam)_{\lam \in M}$ be a family of polynomial
skew products over a fixed base $p$.  
Then
\[
\tbif^k  = \int_{J_p^k} T_{\underline z} \mu^{\otimes k}
\quad \mbox{ and } \quad
\Supp (\tbif^k) = \overline {\cup_{\underline z} 
\Supp T_{\underline z}}.
\]
\end{prop}

\begin{proof}
The case $k=1$ 
follows from 
the explicit
formula for $L_v$
 in \eqref{eq_lyapunov}.
The first formula
in the statement
 is a consequence of the case $k=1$ and the
continuity of the 
potentials of the bifurcation currents $\tbifz$.
The 
continuity of the potentials (in both $z$ and the parameter) 
also implies
that
the currents $T_{\underline z}$
are continuous in $\underline z \in J_p^k$. We can thus apply the general Lemma \ref{lemma_continous_family_currents_metric_space}
below
to the family of 
currents $R_a = T_{\underline z}$ and $a= \underline z \in J_p^k=A$. 
This concludes the proof.
\end{proof}

\begin{lemma}\label{lemma_continous_family_currents_metric_space}
Let $A$ be a compact metric space, $\nu$ a positive measure on $A$ and $R_a$
 a family of positive
closed currents on $\C^N$ depending continuously on $a\in A$. Set $R:=\int_{A} R_{a}\nu (a)$.
Then
\begin{enumerate}
\item the support of $R_a$ depends lower semicontinuously from $a$ (in the Hausdorff topology);
\item the support of $R$ is included in $\overline {\cup_a \Supp R_a }$;
\item for every $a\in \Supp \nu$, we have $\Supp R_a\subseteq \Supp R$.
\end{enumerate}
\end{lemma}

Recall that the current $R= \int_{A} R_a \nu(a)$ is defined by
 the
identity $\langle R, \phi\rangle
=
\int_{A} \langle R_a, \phi\rangle \nu(a)$,
for $\phi$ test form of the right degree.

\begin{proof}
The first property is classical and the second is a direct consequence. Let us prove the last item.
Fix $a\in A$ and take 
$x\in \Supp R_a$. There exists
 an (arbitrarily small) ball $B$ centred at $x$
such that the mass of $R_a$ on $B$ is larger than some $\eta>0$.
 By the continuity of $R_a$, 
 the mass of $R_{a'}$ on $B$
 is larger that $\eta/2$ for \emph{every}
$a'$ sufficiently
close to $a$. In particular, this is true for all $a'$
in a ball $B'$ centred at  $a$.
Since $a\in \Supp \nu$, we have $\nu(B')>\eta'$ for some positive $\eta'$. 
Thus, $R$ has
 mass $> \eta \eta' /2$ on $B$, which in turn gives $x \in \Supp R$.
\end{proof}

\section{Vertical-like hyperbolic sets and IFSs}\label{s:vlifs}

\begin{defi}\label{defi-hyp-vert-like}
Let $f(z)=(p(z),q(z,w))$ 
be a polynomial skew product of degree $\geq 2$
	and let
$H$ be an $f$-invariant hyperbolic set. We say that $H$
is \emph{vertical-like} if there exists $\alpha>0$ such that, for every
$(z,w)\in H$, we have
$df_{(z,w)} (C_\alpha)\Subset C_\alpha$, where 	
	\begin{equation}\label{eq_cone}
	C_\alpha:=\big\{u \in \C^2 :  |\langle u, (0,1)  \rangle | >   \alpha  \|u\| \big\}.
	\end{equation}
\end{defi}

Recall that, given any ergodic measure 
$\nu$ 
supported on a $f$-invariant hyperbolic set $H$, 
by Oseledets theorem
one can associate to $\nu$-almost every $x \in H$
a decomposition of the tangent space $T_x\C^2 = E_1 \oplus E_2$, which is invariant under $f$,
with the property that $\lim_{n\to \infty} n^{-1} \log \|df_x^n (v)\| = \chi_i$ for all $v \in E_i$, where
$\chi_1,\chi_2$ are the Lyapunov exponents of $\nu$.
The hyperbolicity of $H$ implies that the decomposition is continuous in $x$, which in turn implies
that it is also independent of $\nu$.
Since $f$ is a polynomial skew product, we know that one invariant direction must necessarily
coincide with the vertical one. 
Denoting by $E_v= \langle (0,1)\rangle$ and $E_h$ the two fields of directions,
Definition \ref{defi-hyp-vert-like} implies that
  $E_h$
is then uniformly far from the vertical direction.

In the case of a periodic cycle, Definition \ref{defi-hyp-vert-like} can we rephrased
as a condition on the eigenvalues of the differential of the return map at the periodic points.
Although a periodic point is  not an invariant hyperbolic set, we will adopt
 the following notation for simplicity.

\begin{defi}\label{defi-per-vert-like}
Let $f(z)=(p(z),q(z,w))$ 
be a polynomial skew product of degree $\geq 2$
	and let $(z_1,w_1)$ be a 
	  $m$-periodic
	 point for $f$. Let $A:= (p^{m})' (z_1)$ and $B:=(Q_{z_1}^{m})'(w_1)$ be
	 the two eigenvalues of $df^{m}_{(z_1,w_1)}$. We say that
	$(z_1,w_1)$
	  is
	  \emph{vertical-like} if 
	$|B|>|A|$.
\end{defi}

\begin{defi}\label{defi-fibred-box}
Let $F$ be a subset of $J_p$.
We say that a set $A \subseteq F \times \C$ is a \emph{fibred box} if
$A$ is an open subset of $F \times \C$ of the form
$A= \cup_{z\in B} \{z\}\times D_z $ where $B$ is an open subset
 of $F$, 
and 
$D_z\subset \C$ is  a topological disk depending continuously on $z \in B$ and such that
$\mu_z 
(D_z)$ is constant in $z$. 
\end{defi}

   Observe that fibred boxes exist
  since the family of measures $z\mapsto \mu_z$
  is continuous.
  
 \begin{defi}\label{defi-vert-like-ifs}
 Let $H_p$ be a hyperbolic invariant compact subset of $J_p$.
	A \emph{vertical-like IFS} over $H_p$
	is the datum of a
	fibred box
	$W \subset H_p \times \C$
	 and of 
	$m$
	 inverse branches 
	$g_1, \ldots, g_m$ of $f^{-n}$ with $g_m(W) \Subset W$
	(in the relative topology of $J_{H_p}$), 
	and such that:
	\begin{enumerate}[label={\bf (V\arabic*)}]
\item\label{item_vlifs_limit} the limit set is a vertical-like hyperbolic set;
		\item\label{item_vlifs_vertical} for all $1 \leq i \leq m$, there exists $i \neq j$ such that $\pi_z(g_i(W))=\pi_z(g_j(W))$;
		\item\label{item_vlifs_horizontal} there exists $i \neq j$ such that $\pi_z(g_i(W)) \cap \pi_z(g_j(W)) = \emptyset$.
	\end{enumerate}
\end{defi}

Note that due to the skew product structure of $f$, for all $1,\leq i,j\leq m$, we automatically have 
either $\pi_z(g_i(W))=\pi_z(g_j(W))$ or $\pi_z(g_i(W)) \cap \pi_z(g_j(W))= \emptyset$.
We will consider in the following limit sets of vertical-like IFSs, which are then
vertical-like hyperbolic sets (contained in $J_{H_p}$)
 as in Definition \ref{defi-hyp-vert-like} by \ref{item_vlifs_limit}.
Condition \ref{item_vlifs_vertical} ensures that each vertical slice of the limit set is non-trivial
(i.e., it is a Cantor set in $\C$),
and condition \ref{item_vlifs_horizontal} 
ensures that the limit set  is not included in a single vertical fibre.

 In order to prove our main results, we will need that our maps admit
a  vertical-like hyperbolic set.
The following result ensures that this requirement is reasonably mild, and
explains the assumption on $p$ in our Theorems.

\begin{prop}\label{prop:modrep}
	Let $p$ be a polynomial with Julia set not totally disconnected,
	 which is neither conjugated to $z \mapsto z^d$
	nor to a Chebyshev polynomial. Then 
	any polynomial skew product $f$ of the form $f(z,w)=(p(z), q(z,w))$
	admits a vertical-like IFS.
\end{prop}

\begin{proof}
	By a 
	result of Przytycki and Zdunik \cite{przytycki2020hausdorff}
	 (see also \cite{przytycki1985hausdorff,zdunik1990parabolic} for previous results in the connected case), 
	 since $p$ is neither conjugated to $z \mapsto z^d$ nor to 
	a Chebyshev polynomial, there exists a compact hyperbolic invariant set $\tilde H \subset J_p$, with $\delta:=\dim_H \tilde H >1$
	and positive entropy.
	By the general theory of the thermodynamical formalism, there exists a unique ergodic invariant
	probability measure $\tilde \nu$ supported on
	$\tilde H$ that is absolutely continuous with respect to the $\delta$-dimensional Hausdorff measure
	(see for instance \cite{przytycki2010conformal,przytycki2018thermodynamic}).
	By Manning's formula, $L_{\tilde \nu} = \frac{h_{\tilde \nu}}{\delta}$, where
	$L_{\tilde \nu}$ is the Lyapunov exponent of $\tilde \nu$, and $h_{\tilde \nu}$ its metric entropy. Since 
	$\delta>1$ and $h_{\tilde \nu} < \log d$, 
	we deduce that $L_{\tilde \nu}<\log d$.

We now consider the measure $\nu := \int_{\tilde H} \mu_z \ d\tilde \nu(z)$, whose support
is equal to 
 $J_{\tilde H}$.
  The existence of the vertical-like IFS as in the statement will follow
  from the following result.
The proof uses tools
from the thermodynamical formalism
together with quantitative estimates.
  We give it in Appendix \ref{as:proof-lemma-fibered}.

\begin{lemma}\label{lemma-BD-fibered}
For every $\eps>0$ there exists 
a  fibred box $A$
in $\tilde H \times \C$
such that,
 for  all $n$ sufficiently large, 
 the exists at least $3d^n$ 
 $f^n$-inverse branches  $A_i$ of $A$
 compactly
contained in $A$ (for the induced topology on $\tilde H \times \C$)
which are fibred boxes  and
with the property that, for all $i$,
$f^n: A_i \to A$ is injective and
\begin{equation}\label{eq_est_cone}
\frac{1}{n} \log |(p^n)' (x) |< L_{\tilde \nu} + \epsilon
\quad \mbox{ and } \quad
\frac{1}{n}\log |(Q^n_{z})' (x,y)|> L_{v}-\eps
\quad 
\mbox{ for all } (x,y)\in A_i.
\end{equation}
\end{lemma}

Recall that $L_v \geq \log d$, hence $L_v > L_{\tilde \nu}$. 
Let $A, A_i$ be given by 
 Lemma \ref{lemma-BD-fibered} applied 
 with
$\eps < L_v - L_{\tilde \nu}$  and
  $n$ sufficiently large.
 Since the entropy of $\tilde H$ is smaller than $\log d$,
  up to removing a small number of $A_i$'s (bounded by $d^n$)
we can assume that for every $j$ there exists $i\neq j$ such that
$A_i$ and $A_j$ have the same projection on the first component, giving \ref{item_vlifs_vertical}.
The number of remaining $A_i$'s is still bounded below by $2d^n$.
Since at most $\sim d^n$ of them can share the same projection on the first coordinate, this also proves \ref{item_vlifs_horizontal}.
The assertion follows since the inequalities in Lemma \ref{lemma-BD-fibered}
imply that the limit set is a vertical-like hyperbolic set, giving \ref{item_vlifs_limit}.
\end{proof}

\section{Higher bifurcations: an analytic criterion}\label{s:higher_current}\label{s:beg}

In this section we establish the following 
technical result,
which gives an analytic 
sufficient condition for a point to lie in the support
of the
higher bifurcation currents.
Recall that, given  a simple critical point $c(\lam)$ 
for $q_{\lam,z_0}$
and a
repelling point $r(\lam)$
for $f_\lam$,
 we denote by
  $M_{(z_0,c),r, n_0}$ 
  the analytic subset of $M$ given by the equation
  $f^{n_0}_{\lam} (z_0, c(\lam))=r(\lam)$.

\begin{prop}\label{prop_sufficient_condition_bifk}
Let $(f_\lam)_{\lam\in M}$
 be a holomorphic family of polynomial skew products over a given base $p$.
Let $\lam_0\in M$ and  
 $z_1, \dots, z_k\in J_p$ 
 satisfy the following properties:
\begin{enumerate}
\item there exist
 simple critical points
  $c_i$
   for $q_{\lam_0,z_i}$
such that $r_i := f_{\lam_0}^{m_i}(z_i,c_i)$ is a repelling periodic 
point
for $f_{\lam_0}$;
\item $\codim \cap_{i=1}^k  M_{(z_i,c_i),r_i, m_i} =k$.
\end{enumerate}
Then $\lam_0 \in \Supp \tbif^k (M)$.
\end{prop}

In the case of families of rational maps, this result
is due to Buff-Epstein \cite{buff2009bifurcation}, using transversality arguments.
In \cite{gauthier2012strong}, Gauthier
uses different arguments that only require that
the intersections are proper, as is the case in Proposition \ref{prop_sufficient_condition_bifk}. A more general
condition (called the \emph{generalized large scale condition}) was introduced in \cite{astorg2019collet} as
a sufficient condition for a point to lie in the support of $\tbif^k$ (for a family of rational maps). 
 We give an adapted version of this notion in our non-autonomous
setting, and deduce 
that a parameter $\lam_0$
as in the statement satisfies such condition. This will prove Proposition \ref{prop_sufficient_condition_bifk}.

In the following we assume that $z_1, \dots, z_k \in J_p$ and that $c_j(\lam)$ are holomorphic maps
such that $c_j(\lam)$ is a critical point for $q_{\lam,z_j}$ for all $\lam\in M$. We denote by $\underline c : M\to \C^k$ the map
$c(\lam)=(c_1(\lam), \dots, c_k(\lam))$.
For a $k-$uple $\underline n := (n_1, \dots, n_k)$, we
define
\begin{equation}\label{eq_xi}
\xi_{n_j}^{j} (\lam) := Q^{n_j}_{\lam,z_j} (c_j(\lam))
\quad
\mbox{ and }
\quad
\Xi_{\underline n}^{\underline c} (\lam) := (\xi_{n_1}^1 (\lam), \dots \xi_{n_k}^k (\lam)).
\end{equation}
Notice that $\Xi_{\underline n}^{\underline c}: M \to \C^k$.
We denote by $C_j$ the graph of $c_j$ in $M\times \C$
and by
 $V_{\underline n}$
the graph of $\Xi_{\underline n}^{\underline c}$ in $M\times \C^k$.
We also write $\abs{\underline n} :=n_1 + \dots + n_k$
for a $k$-uple $\underline n$ as above.

\begin{defi}[Fibred large scale condition]\label{defi_large_scale}
We say that $\lam_0\in M$
satisfies the \emph{fibred large 
scale condition} for the critical points $(z_1, c_1), \dots, (z_k,c_k)$
if there exist $z'_1, \dots z'_k\in J_p$, 
disks $D_1, \dots, D_k\subset \C$ with $D_i \cap J_{z'_i} \neq \emptyset$,
a sequence ${\underline n}_l= (n_{l,1}, \dots, n_{l,k})$ of $k-$uples
with $n_{l,i}\to \infty$
 and a nested sequence of open subsets 
 $\Omega_l$ 
 such that
\begin{itemize}
\item
$\cap_{l} \bar \Omega_l = \{\lam_0\}$, and
\item $\Xi_{{\underline n}_l}^{{\underline c}} \colon \Omega_l \to D_1 \times \dots \times D_k$ is a proper surjective map.
\end{itemize}
\end{defi}

\begin{prop}\label{prop_scale_then_bif}
Let $\lam_0 \in M$
satisfy the fibred
large scale condition for some points $(z_1,c_1), \dots, (z_k,c_k)$
with $q'_{z_j} (c_j)=0$ for every $j$ 
and such that the $z_j$ are preperiodic for $p$.
Then $\lam_0 \in
\Supp {\tbif}_{z_1}\wedge \dots \wedge {\tbif}_{z_k}$.
\end{prop}

\begin{proof}
The proof 
follows the same line as that of 
\cite[Theorem 3.2]{astorg2019collet}. We give
 here the main steps.

First of all, it is enough to prove
the statement in the assumption that the dimension of
$M$
is equal to $k$, see \cite[Lemma 6.3]{gauthier2012strong}.
For every $\underline n = (n_1, \dots, n_k)$ with $n_j\geq 0$, we
define the map
\[
\begin{aligned}
F_{\underline n}
 \colon & M \times \C^k & \to & \quad M \times \C^k \\
&(\lam, w_1, \dots, w_k) & \mapsto & \quad (\lam, Q_{\lam,z_1}^{n_1} (w_1), \dots, Q_{\lam, z_k}^{n_k} (w_k)).
\end{aligned}
\]
and we denote by $\tilde \pi_j\colon M \times \C^k \to M \times \C$ 
the projection
$(\lam, w_1, \dots, w_k)\mapsto (\lam, w_j)$.
One can prove that, for every $\underline{n}$ as above
and Borel set $\Omega \subseteq M$,
\[
\begin{aligned}
{\tbif}_{z_1}\wedge \dots \wedge {\tbif}_{z_k} (\Omega) 
&= d^{- \abs{n}} \int_{\Omega \times \C^k} 
F_{\underline n}^*
 \Big(
 \bigwedge_{j=1}^k \tilde \pi_j^*  (dd^c_{\lam, w} G_{\lam} (z'_j,\cdot) )
 \Big)
 \wedge 
 \Big[
 \bigcap_{j=1}^k C_j
 \Big]\\
&= d^{- \abs{n}} \int_{\Omega \times \C^k} 
\Big(
\bigwedge_{j=1}^k \tilde \pi_j^*  (dd^c_{\lam, w} G_{\lam} (z'_j,\cdot) ) \Big)
 \wedge [V_{\underline n}],
\end{aligned}\]
see \cite[Lemma 3.3]{astorg2019collet}.
Moreover, 
with $\Omega_l$ and $\underline n_l$
as in the statement,
we also have (see \cite[Lemma 3.4]{astorg2019collet}) that
\[
\liminf_{l\to \infty}
 \int_{\Omega_l \times \C^k}
  \Big(\bigwedge_{j=1}^k \tilde \pi_j^*  (dd^c_{\lam, w_j} G_{\lam} (z'_j,\cdot) ) \Big) 
  \wedge [V_{\underline n_l}]
\geq \prod_{j=1}^k   (dd^c_{w} G_{\lam_0} (z'_j, \cdot) )(D_j).
\]
We use in this step the second assumption in
Definition
\ref{defi_large_scale}.
The right hand side of the last expression is strictly positive by the assumption that $D_j \cap J_{z'_j} \neq \emptyset$. 
This implies that  ${\tbif}_{z_1}\wedge \dots \wedge {\tbif}_{z_k} (\Omega_l) >0$, for
a sequence of integers $l$ going to infinity. Hence $\lam_0 \in \Supp {\tbif}_{z_1}\wedge \dots \wedge {\tbif}_{z_k}$, as desired.
\end{proof}

We can now prove Proposition \ref{prop_sufficient_condition_bifk}.

\begin{proof}[Proof of Proposition \ref{prop_sufficient_condition_bifk}]
By Proposition \ref{prop_decomposition} it is enough to prove that
 $\lam_0 \in \Supp {\tbif}_{z_1}\wedge \dots \wedge  {\tbif}_{z_k}$.
By Proposition \ref{prop_scale_then_bif} it 
is thus enough to prove that any $\lam_0$ as in the statement satisfies
 the 
 fibred large scale condition above. 
 We can also assume that the dimension of $M$ is $k$.

Denote by $s_i$ the period of the repelling point $r_i$
and set $\underline s = (s_1, \dots s_k)$.
Set $r_i =: (z'_i,r'_i)$ and similarly let
$r_i (\lam) = (z'_i, r'_i (\lam))$ be the motion of 
$r_i$ in a neighbourhood of $\lam_0$ as a periodic point.
Fix $\eta>0$ and
an open neighbourhood $\Omega$
of $\lam_0$ such that the following properties hold:
\begin{enumerate}
\item for all $r_i$ 
 as in the statement, 
 $r'_i(\lam)\in \D (r'_i, \eta/10)$
 for all $\lam\in \Omega$;
 \item  
 for every $i$ and every $\lam\in \Omega$,
  the map $Q^{s_i}_{\lam,z'_i}$ is uniformly expanding on
   $\D (r'_i (\lam), \eta)$ (with expansivity factor uniform in $\lam$).
  \end{enumerate} 
  Observe
  that, for all $\lam \in \Omega$, we have $\D(r'_i, \eta/2)\subset \D (r'_i (\lam), \eta)$. 
  We
  set
\[A_0 :=\{ (\lam, w_1, \dots, w_k) \in \Omega \times \C^k \colon w_i \in \D (r'_i(\lam), \eta)\}.\]
We denote by $g_{\lam,i}\colon \D (r'_i (\lam), \eta)\to \C$
  the inverse branch of $Q^{s_i}_{\lam,z'_i}$ such that $g_{\lam,i} (r'_i (\lam))= r'_i (\lam)$
  and  by $\underline G: A_0 \to \Omega \times \C^k$
  the inverse branch of $F_{\underline s}$
which agrees  on $A_0$ with the $g_{\lam,i}$ as above.
For $l \in \N$, 
  we set $A_l := \underline G^l (A_0)$. 
  Observe that
$A_l$
 shrinks (exponentially) with $l\to \infty$ to the graph of the product map
$\lam\mapsto (r'_1(\lam), \dots, r'_k (\lam))$.
  
  Consider the map $\Phi'\colon \C^k \to \C^k$ defined by
  \[
  \Phi' ( w_1, \dots, w_k) =  (w_1- r_1 (\lam), \dots, w_k - r_k (\lam))\]
  and set $\Phi (\lam, w_1, \dots, w_k):= (\lam, \Phi'(w_1, \dots, w_k))$.
  Observe that
   $\Phi (\lam, r_1 (\lam), \dots, r_k (\lam))= (\lam, 0, \dots, 0)$. We denote
    $B_0 := \Phi (A_0) = \Omega \times \D(0, \eta)^k$
 and similarly
 set $B_l := \Phi (A_l)$.  
 We will also need the projections
  $\pi_M, \underline{\pi}$ of $ M\times \C^k$
  on
  $M$ and
   on
   $\C^k$, respectively.

For every $\underline n = (n_1, \dots, n_k)$ consider the map $H_{\underline n} \colon \Omega \to \C^k$ given by
$H_{\underline n} := \Phi' \circ \Xi^{\underline c}_{\underline n}$, 
where $\Xi^{\underline c}_{\underline n}$ is defined in \eqref{eq_xi}.
We claim that the map $H_{\underline m}$ is open in a neighbourhood of $\lam_0$,  
  where $\underline m = (m_1, \dots, m_k)$.
  By \cite[\S 3.1.2 and \S 5.4.3]{grauert2012coherent}, it is enough to check that
  the point $\lam_0$ is isolated in 
 $ (H_{\underline m})^{-1}  H_{\underline m} (\lam_0)$. 
 This is precisely given by the second assumption in the statement. 
 The same
 assumption
and the fact that the $q_{\lam,z}$'s are open imply that, for any $l\in \N$, we also have
 $\codim \cap_{i=1}^k  M_{(z_i,c_i),r_i, m_i+ls_i} =k$. The argument above implies that
 also the maps $H_{\underline n_l}$ are open, where $\underline n_l := (m_1 + ls_1 , \dots , m_k + l s_k)$.
 
 By restricting if necessary the $\Omega$ as above, 
we see that the graph
$\Gamma_0$
 of the map
  $H_{\underline m}$
 is of dimension $k$ in $B_0$.
We set 
$\Omega_l := \pi_M (\Gamma_0 \cap B_l)$. The $\Omega_l$'s 
are then open.
Since $B_l$ shrinks 
with $l$ to the constant graph $\{(\lam, 0, \dots, 0)\}$, 
 we also have that
$\Omega_l$ shrinks to $\{\lam_0\}$ as $l\to \infty$. 

Set $D_i := \D (r'_i, \eta/4)$, let $\Gamma_l$ be
the graph of  $H_{\underline n_l}$ on $\Omega_l$
  (which, by the above, is also $k$-dimensional)
 and recall that 
 $V_{\underline n_l}$
denotes  the graph of
$\Xi^{\underline c}_{\underline n_l}$.
 To conclude it is enough to prove
that, for all
$l\in \N$, 
$\underline{\pi} (\pi_M^{-1} (\Omega_l) \cap V_{\underline n_l} )  \supset \prod_{i=1}^k D_i$.
We will use the 
  following fact. 
  
  \begin{fact*}
Let $\Omega'\Subset  \Omega$ be 
an open subset.
Let
 $W_v, W_h$ be two non-empty $k$-dimensional closed analytic subsets of 
 $B_0$
 with
  $\pi_M (W_v) \subset  \Omega'$
and $\underline{\pi} (W_h)\subset \D (0, \eta/2)^k$.
Then $W_h\cap W_v \neq \emptyset$.
  \end{fact*} 

Recall that $r'_i(\lam)\in \D(r'_i, \eta/10)$ for all $i$ and $\lam \in \Omega$. Hence 
the Fact, 
applied with $\Omega'=\Omega_l$, $W_v = \Gamma_l$ and $W_h= \{(\lam, y_1 -r'_1(\lam), \dots, y_k - r'_k (\lam))\}$, 
implies that, 
for any $\underline y = (y_1, \dots, y_k) \in\prod_{i=1}^k \D(r'_i,\eta/4)$, 
there exists a $\lam \in \Omega_l$
such that
\[
Q^{m_i + l s_i}_{z_i} (c_i(\lam))-r'_i (\lam)=
y_i -r'_i(\lam) \mbox{ for all } 1\leq i \leq  k.
\]
This implies that $\underline{\pi} (V_{\underline n_l}) \supset\prod_{i=1}^k D_i$, as desired. The proof is complete.
\end{proof}

\begin{remark}
As is the case in \cite{astorg2019collet}, it is enough to make a weaker assumption in
 Proposition \ref{prop_sufficient_condition_bifk}, namely
that the critical orbits fall in the motion of some hyperbolic set. The proof is slightly more
involved in that situation  (as is the case in \cite{astorg2019collet}).
We prefer to state only the simple criterion based on repelling periodic orbits since this simpler version
will be enough to deduce our main result.
\end{remark}

\section{Creating multiple bifurcations: a geometric method}\label{s:machine}

In this section we develop our method
to construct multiple bifurcations
(in the form of Misiurewicz parameters)
starting from a simple one. In the next section we will ensure the applicability
of this method. First, let us introduce 
the following definition.

\begin{defi}\label{def:goodmis}
	Let $(f_\lambda)_{\lambda \in M}$ be a 
	holomorphic family of polynomial skew products
	over a fixed base polynomial $p$.
	We say that $M$ 
	is a \emph{good Misiurewicz family}, or that $M$
	has a \emph{persistently good Misiurewicz relation}
	$f_\lam^{n_0}(z_0,c(\lam))=(z_1,w_1(\lam))$
	if the Misiurewicz relation $f_\lam^{n_0}(z_0,c(\lam))=(z_1,w_1(\lam))$
	(where 
	$(z_1, w_1(\lam))$ is a repelling periodic point of period $m$ for $f_\lam$) is 
	persistent in $M$,
	and if 
	moreover
	\begin{enumerate}[label={\bf (G\arabic*)}]
		\item\label{item_good_non_constant} the vertical eigenvalue $B(\lambda):=(Q_{\lam,z_1}^m)'(w_1(\lambda))$ is non-constant on $M$;
		\item\label{item_good_vertical_like} for all $\lam \in M$, $(z_1,w_1(\lam))$ is vertical-like;
		\item\label{item_good_znotcrit} $(p^{n_0})'(z_0) \neq 0$ and $z_0 \notin \{p^i(z_1), 1 \leq i \leq m \}$;
		\item\label{item_good_simple} for all $\lam \in M$, $c(\lam)$ is a simple root of $q_{\lam,z_0}'$;
		\item\label{item_good_nontangent} for all $\lam$, if $L_\lam$ denote the unique component of $\crit(f_\lam)$ passing through $(z_0,c(\lam))$, then $f_\lam^{n_0}(L_\lam)$ is regular at $(z_1, w_1(\lam))$ and is not tangent to an eigenspace of $df_\lam^m(z_1,w_1(\lam)) $.
	\end{enumerate}
	A parameter $\lam_0 \in M$ satisfying all the conditions above 
	will be called a \emph{good Misiurewicz parameter}.
\end{defi}

Observe that a good Misiurewicz family in $\skpd$
 is, in general,
 an open subset of an algebraic hypersurface of $\skpd$.
The next Proposition is the key technical result of our argument.

\begin{prop}\label{prop:secmis}
	Let $(f_\lambda)_{\lambda \in M}$
 be a holomorphic family of polynomial
	skew products 
	over a fixed base polynomial $p$ and with a persistently good Misiurewicz relation
	 $(z_1,w_1 (\lam)):=f_{\lambda}^{N_0}(z_0,c_0 (\lam))$.
	 There exists 
	a dense subset
	$S \subset M$ such that
 for all $\lam_\infty \in S
$
and
	for every
 $(z'',w'')$
	 repelling periodic point in the limit set of a
	vertical-like IFS for $f_{\lam_{\infty}}$, 
	 there exists a sequence $\lam_n \to \lam_\infty$
	such that $f_{\lam_n}$ has a Misiurewicz relation
	of the form $f_{\lam_n}^{N_n}( y_n,  c_n(\lam_n)) = (z'', w''(\lam_n))$
	(where $(z'',w''(\lam))$ is the  holomorphic motion as repelling periodic point  of $(z'',w'')$ in a neighbourhood
	of $\lam_\infty$)
	which is non-persistent on $M$ and satisfies \ref{item_good_vertical_like}, 
\ref{item_good_znotcrit},	 \ref{item_good_simple},
	and \ref{item_good_nontangent}   on a neighbourhood of 
	$\lam_n$ in $M_{(y_n, c_n(\lam_n)), (z'', w''(\lam_n), N_n)} \subset M$. 
\end{prop}

\begin{cor}\label{cor:bifm=m}
Let $(f_\lambda)_{\lambda \in M}$ be a holomorphic family of polynomial
	skew products 
	with a persistently good Misiurewicz relation. Then,
	$\Bif(M)=M$.
\end{cor}

\begin{proof}
	The assertion follows from the fact that Misiurewicz parameters belong to the bifurcation locus, see	
	\cite{bbd2015,b_misiurewicz}.
\end{proof}

The remaining part of this section is devoted to proving Proposition \ref{prop:secmis}.
 We start defining the set $S$.
\begin{defi}\label{defi-S}
	Let $M$ be a good Misiurewicz family, and let 
	$f_\lam^{n_0}(z_0,c(\lam))=(z_1,w_1(\lam))$ be a
	persistent Misiurewicz relation satisfying the requirements 
	of Definition \ref{def:goodmis}.
	
	We define the set $S \subset M$ to be the set of $\lambda_\infty \in M$ for which 
	each of the following properties holds: 
	\begin{enumerate}[label={\bf (S\arabic*)}]
		\item\label{item_s_db}	
		$d_\lam B (\lam_\infty)\neq 0$;
		\item\label{item_s_resonance} $\log B(\lam_\infty) \notin \R \log A$.
	\end{enumerate}
\end{defi}

Note that
$S$ is 
open and dense
in $M$.
From now on, we fix an arbitrary $\lambda_{\infty} \in S$, and we choose a one-dimensional disk in local coordinates
in $M$ transverse to a level set of $B$ in which $\lambda_{\infty}=0$ (hence $\frac{dB}{d\lambda} (0) \neq 0$).
The proof of Proposition \ref{prop:secmis} will mostly use local arguments in phase space. Therefore, 
we will work in local linearizing coordinates near $(z_1,w_1)$; in particular, in the rest of this 
section we will take $(z_1,w_1)=(0,0)$, and we will 
assume that $m=1$
(which we can do up to passing to an iterate).

By item \ref{item_s_resonance}
of the definition of $S$, there are no resonances between the eigenvalues
of this fixed point $(0,0)$ for $\lam$ close to $\lam_\infty=0$. We may therefore
assume that the  fixed point $(0,0)$ is linearizable for $f_\lam$; moreover the linearizing
map can be chosen to depend holomorphically on the parameter.
More precisely, we can fix a 
neighbourhood $U$
of $(0,0)$
such
that these linearizing coordinates are defined 
for $(z,w) \in U$ for all $f_\lam$
with $|\lambda|$
small enough.
So $f_\lambda$ acts in those coordinates as 
the linear map $(z,w) \mapsto (Az, B(\lambda) w)$.

It follows from the Implicit Function Theorem and 
\ref{item_good_simple} that there is a unique component of $\crit(f_\lam)$ passing through $(z_0,c(\lam))$, and 
that this component is smooth and can be locally described as a graph of the form $w=\tilde \beta(z,\lam)$, 
for some holomorphic germ $\tilde \beta$. Setting $\beta(z,\lam):=Q_{\lam,z}^{N_0}(\tilde \beta(z,\lam))$, 
this implies that 
the graph $w=\beta(z,\lam)$ is a local parametrization of a component 
$L_\lambda$ of $f_\lambda^{N_0}(\crit(f_\lambda))$. 
The assumption 
\ref{item_good_nontangent} 
and our choice of local coordinates
imply that 
the holomorphic map $z \mapsto \beta(z,\lambda)$ is not constantly equal to $0$, and 
moreover  that $\beta_1 := \frac{\partial \beta}{\partial z}(0,0) \neq 0$.

\begin{lem}\label{lem:renorm}
	Let $K\subset \C^*$
	be a compact set.  Let  $(z_k)_{k\in\N}$ be a sequence in $J_p$
	such that 
	$z_k \to 0$, and $(m_k)_{k \in \N}$ be a sequence of integers such that 
	$z_k A^{-m_k} B^{m_k}\in K$
	for all $k$.
	Set $\phi_k(\lam):=\beta(z_k A^{-m_k},\lam) B(\lam)^{m_k}$. 
	Then the sequence $(\phi_k)_{k\in \N}$ is not normal at $\lam=0$.
\end{lem}

\begin{proof}
	Let us compute the derivative of $\phi_k$ at $0$:
	\begin{align*}
		\frac{d \phi_k}{d \lambda}(0) &=\frac{\partial \beta}{\partial \lam}(z_k A^{-m_k},0) B^{m_k} + \beta(z_k A^{-m_k},0) B'(0) m_k B^{m_k-1} \\
		&=\mathcal{O}\left( z_k A^{-m_k} B^{m_k} \right) 
		+ \beta_1 z_k A^{-m_k} B'(0) m_k B^{m_k-1} + \mathcal{O}\left( z_k^2 A^{-2 m_k} B^{m_k}	\right)\\
		& + \mathcal {O} \left( z_k^2 A^{-2 m_k} B^{m_k} m_k	 \right).
	\end{align*}
Since  $\beta_1
 \neq 0$, by the choice of $m_k$, we have
	\begin{align*}
		\frac{d \phi_k}{d \lambda}(0) = \beta_1 z_k A^{-m_k} B'(0) m_k B^{m_k-1} + \mathcal {O}(1)   \asymp m_k,
	\end{align*}
	hence $\lim_{k \to +\infty} |\frac{d\phi_k}{d\lambda} (0)|=+\infty$. This proves 
	the non-normality of $(\phi_k)$ at $0$.
\end{proof}

\begin{lem}\label{lem:goodperdense}
	 Let $H_0$ be the limit set of a vertical-like IFS
	 for $f_{0}$.
	  Let $(z',w'), (z'',w'') \in K$
	  be periodic points
	   and let $U$ be an open neighbourhood of $(z',w')$.
	 There exist $w_1\neq w_2\in \C$, both distinct from $w'$ and
	 with $(z',w_1), (z',w_2)\in H_0\cap U$, and sequences $(z_k,w_{k,i})$ (with $1 \leq i \leq 2$) 
with $\lim_{k \to +\infty} (z_k, w_{k,i})=(z', w_i)$
and
	 such that, for all $k \in \N$
	 and $i \in \{1,2\}$,
	 \begin{enumerate}
	 	\item there exists $n_k \in \N$ such that $f^{n_k}(z_k, w_{k,i})=(z'',w'')$;
	 	\item $(z_k, w_{k,i}) \in H_0 \cap U$;
	 	\item $z_k \neq z'$. 
	 \end{enumerate}
\end{lem}

\begin{proof}
	First, let $(z',w_1)$ and $(z',w_2)$ be two periodic points in $H_0 \cap U$ such that $w', w_1$, 
	and $w_2$ are pairwise distinct.
	Such points exist since $H_0 \cap (\{z'\}\times \C)$ 
	contains (the image of) the limit set of a non-trivial IFS in $\C$ by
	the condition \ref{item_vlifs_vertical} in
	 Definition \ref{defi-vert-like-ifs}.
	 We can also assume that $(z',w_1)$ and $(z',w_2)$ have the same period.

	 There exists a finite sequence $g_{i_1}, \dots, g_{i_{\ell_1}}$ with the property that 
	  $(z', w_1)$ is the unique fixed point of the finite composition $G_1 := g_{i_1} \circ \dots \circ g_{i_\ell}$.
	  Similarly, $(z',w_2)$ is the unique fixed point of a finite composition $G_2 := g_{j_1} \circ \dots \circ g_{j_{\ell_2}}$.
	  Since $(z',w_1)$ and $(z',w_2)$ have the same period, we can assume that $\ell_1 = \ell_2 = \ell$.
	  Moreover, since $(z',w_1)$ and $(z',w_2)$ belong to the same vertical fibre,
 the maps $G_1$ and $G_2$ agree on the first coordinate.

We now construct the sequences $(z_k, w_{k,i})$. 
We first assume that $z''$
does not belong to the orbit of $z'$ under the base polynomial $p$.
In this case, given $k_0 \in \N$ it is enough to set
\[
(z_{k,i} , w_{k,i}) := G_i^{k+k_0} (z'',w'') \mbox{ for all } k\geq 1.
\]
Since $G_1$ and $G_2$ agree on the first coordinate, we have $z_{k,1}= z_{k,2}$ for all
$k$.  We set $z_k := z_{k,1}=z_{k,2}$ for all $k$.
For $i \in \{1,2\}$, the sequence $(z_{k} , w_{k,i})$
 converges to $(z',w_{i})$
as $k\to \infty$ by the definition of $G_i$.
When $k_0$ is taken sufficiently large, all the points in such sequences belong to $U \cap H_0$.
Finally, we have $z_{k,i}\neq z'$ for all $k$ since by assumption  $z''$ is not in the orbit of $z'$.

Suppose now that $z''$ belongs to the orbit of $z'$. 
Since backwards preimages of $(z'', w'')$
are dense in $H_0$, there exists at least one such preimage $z'''$
not in the orbit of $z'$. We choose $w'''$ so that $(z''', w''')\in H_0$ and
$(z'',w'')$ is in the orbit of $(z''',w''')$.
It is enough to apply the above argument to $(z''', w''')$
instead of $(z'',w'')$. The proof is complete.
\end{proof}

We are now ready to prove Proposition \ref{prop:secmis}.

\begin{proof}[Proof of Proposition \ref{prop:secmis}]
	We are working in the setting described after Definition \ref{defi-S}.
	We fix a periodic point $(z'',w'')$
	in the limit set $H_0$
	 of a vertical-like 
	IFS as in the statement. Observe that $(z'',w'')$ is vertical-like.
	We can choose $(z'',w'')$ with $z''$ not in the post-critical set of $p$.
	We also denote by $H_\lam$ the holomorphic motion 
	of $H_0$ as hyperbolic set in a neighbourhood of $\lam=0$.

	We let $(z_k, w_{k,i})$ be the sequences of preimages of $(z'',w'')$
	 given by Lemma \ref{lem:goodperdense}
	applied to $(z',w')=(0,0)$ and $(z'',w'')$.
	 Since all of these
	  points belong to $H_0$, 
	they all move holomorphically as $(z_k, w_{k,i}(\lam))$ over a common domain in parameter space. 
	Moreover, by the Definition \ref{defi-hyp-vert-like} and continuity,
	we may fix a vertical cone
	$C_{\alpha_0}:=\{u \in \C^2 :  |\langle u, (0,1)  \rangle |> \alpha_0  \|u\| \}$
	 such that,
	 for all $\lam$ in a neighbourhood of $0$ and $(x,y)\in H_\lam$,
	  we have
$(df_\lam)_{(x,y)} (C_{\alpha_0}) \Subset C_{\alpha_0}$. This implies that the non-vertical
Oseledets direction are uniformly bounded away from the vertical direction.

Fix $\eps>0$. We want to prove that there exists $\lam \in \D(0,\eps)$
	and $i_0 \in \{1,2\}$
	 such that 
	$(z_k, w_{k,i_0}(\lam))$
	(and hence $(z'',w'' (\lam))$)
	 is (non-persistently) in the post-critical set of $f_\lam$.
	To that end, observe that 
	\begin{equation}
		f^{m_k}(z_k A^{-m_k}, \beta(z_k A^{-m_k}, \lam))=(z_k,  \beta(z_k A^{-m_k}, \lam) B(\lam)^{m_k} ) = (z_k, \phi_k(\lam))
	\end{equation}
	is a post-critical point; it is therefore enough 
	to prove that there exist  sequences $\lam_k \to 0$
	and $(i_k) \in \{1,2\}^{\N}$ such that $w_{k,i_k}(\lam_k)=\phi_k(\lam_k)$.

	By Lemma \ref{lem:renorm}, the sequence $(\phi_k)_k$ is not normal at $\lam=0$.
	Therefore, by Montel theorem, the sequence of the graphs of the $\phi_k$'s 
	cannot avoid both those of $w_{k,1}$ and $w_{k,2}$. Hence there exist 
	$\lam_k \to 0$ and $(i_k) \in \{1,2\}^\N$ such that $\phi_k(\lam_k)=w_{k,i_k}(\lam_k)$.
	Up to a subsequence, we can assume that $i_k$ is constant.
	By Lemma \ref{lem:renorm} and the Definition
	\ref{defi-vert-like-ifs} of a vertical-like IFS, 
	the Misiurewicz relation constructed as above
	satisfies \ref{item_good_vertical_like}. 
By the choice of $(z'',w'')$ at the beginning of the proof,  this relation
	satisfies the first condition in \ref{item_good_znotcrit}. 
By taking only $k$ large enough, the second part of the condition is satisfied, too.
	Condition \ref{item_good_simple}
	holds since, by assumption, $c(0)$ 
	is a simple critical point for $q'_{0,z_0}$, hence the same is true for small $\lam$ and $z$
	close to $z_0$. 	
It remains to prove that the relation satisfies \ref{item_good_nontangent}.

	By the definition \eqref{eq_cone} of $C_\alpha$ and the choice of $\alpha_0$ at the beginning of the proof,
	 it is enough to prove that for all $k$ large enough, the tangent
	space to the component of the postcritical set
	passing through $(z_k, w_{k,i_0}(\lam_k))$ and giving the Miriurewicz relation above 
	lies in $C_{\alpha_0}$.
	The branch of the postcritical set is locally given by the equation 
	\begin{equation*}
		w=\beta(z A^{-m_k}, \lam_k) B(\lam_k)^{m_k},
	\end{equation*}
	and therefore its tangent space is generated by the vector 
	\begin{equation*}
		u_k:=\left(1, \frac{\partial}{\partial z}_{|z=z_k} \beta(z A^{-m_k}, \lam_k)   ) B(\lam_k)^{m_k}  \right).
	\end{equation*}
	Since
	\begin{align*}
		 \frac{\partial}{\partial z}_{|z=z_k} \beta(z A^{-m_k}, \lam_k)   ) B(\lam_k)^{m_k} 
		 &=\frac{\partial \beta}{\partial z}(z_k A^{-m_k}, \lam_k) \left(\frac{B(\lam_k)}{A}\right)^{m_k} \\
		 &\sim_{k \to +\infty} \beta_1 \cdot \left(\frac{B(\lam_k)}{A}\right)^{m_k},
	\end{align*}
	and  $\beta_1
	 \neq 0$, it follows that,  for all $k$ large enough, $u_k$ belongs to $C_{\alpha_0}$. The proof is complete. 
\end{proof}

\section{Proof of the main results}\label{s:proof_teo}

In this section, we will first apply inductively Proposition \ref{prop:secmis} in order to prove our main Theorem
\ref{teo_main_no_current}, and then deduce from this result and its proof the other results in the Introduction.
We start with a few required lemmas.

\begin{lem}\label{lem:landingonrep}
	Let $(f_\lam)_{\lam \in M}$ be a holomorphic family of polynomial skew products
over a fixed base $p$ and of a given degree $d\geq 2$, 
	and take $\lam_0 \in \Bif(M)$. 
	There exists a finite set $E \subset J_p$ such that for all $z_1 \in J_p \backslash E$,
	if $(z_1,w_1)$ 
	is any repelling periodic point of $f_{\lam_0}$ (which we locally follow as $(z_1, w_1(\lam))$)
	then, for every $n_0 \in \N$,
	 arbitrarily close to $\lam_0$ there exists  $\lam_1 \in M$ such that $f_{\lam_1}$ has a Misiurewicz relation
	of the form $f_{\lam_1}^n(z,c)=(z_1, w_1(\lam_1))$
 with $(p^n)'(z) \neq 0$ 
  and $n\geq n_0$.
\end{lem}

\begin{proof}
First of all, let us define $E$ as the union of all repelling periodic points in the postcritical set. 
This set is finite. We fix any repelling periodic point  $z_1\notin E$.
By \cite[Proposition 3.5]{astorg2018bifurcations}, we may find 
$\lam'_0$
arbitrarily close to $\lam_0$
 for which a critical point of the form $(y,c)$ 
	 is active, where $y$ is in the strict backward orbit of $z_1$ by $p$. By Montel
	 theorem, we can further slightly perturb $\lam'_0$  to a $\lam_1$ with the property that
	 some iterate of $(y,c)$ 
	 by $f_{\lam_1}$
	 coincides with
	$(z,w_1(\lam_1))$.
	 This completes the proof. 	
\end{proof}

\begin{lem}\label{lem:multtrans}
	Take $\lam_0 \in \skpd$, let $z \in J_p$ be a periodic point of period $m>d$, and
	 $(z, w_i)$ ($1 \leq i \leq D_d= \dim \skpd$) 
	denote 
	a collection of
	repelling periodic points of $Q_{\lam_0,z}^m$ of different 
	periods, which we follow locally as $(z, w_i(\lam))$ over a domain $U \subset \skpd$ containing $\lam_0$. 
	Let $\rho_i(\lam)$ denote their vertical multipliers, 
	and
	let ${\rho}: U \to \C^{D_d}$ denote the map ${\rho}: \lam \mapsto (\rho_i(\lam))_{1 \leq i \leq D_d}$.
There exists an analytic hypersurface $R \subset U$ such that for all
	$\lam \in U \backslash R$,  the differential $d{\rho}_{\lam}$ is invertible.
\end{lem}

\begin{proof}
	First we claim that, 
since	  the period $m$ of $z$ satisfies $m > d$, 
 the family of the first returns $(Q_{\lam,z}^m)_{\lam \in \skpd}$
	can be mapped to
	an algebraic subfamily of pure dimension $D_d= \dim \skpd$ 
	in the space $\pol(d^m)$ of monic centred 
	degree  polynomials of degree $d^{m}$
	(the fact that the image is given by monic centred polynomials follows from the
	parametrization of $\skpd$ given in Lemma \ref{lemma_param_skpd}).
	 Indeed, consider first the map $\phi_z: \skpd \to \pd^m$ 
	defined by $\phi_z(\lam)=(q_{\lam,z_i})_{1\leq i \leq m}$, where $z_i := p^i (z)$.
Since $m>d$ and the coefficients of $q_{\lam,z_i}$ are given by polynomials in $z_i$ of degree at most $d$,
	the map $\phi_z$ is injective.
	Then, consider the map $C: \pd^m \to \pol(d^m)$ defined by $C(q_m, \ldots, q_1)=q_m \circ \ldots \circ q_1$.

	\begin{claim}
	The differential of $C$ at $(w^d, \ldots, w^d)$ is injective.
	\end{claim}
	\begin{proof}
	For $\eps >0$, consider the polynomials $q_i = w^d + \eps r_{i}$, with $r_i$ polynomials in $w$ of 
	degree $\leq d-2$.
	For every $j\leq m$, we also set
	$Q_j (\eps,w ):= q_{j} \circ \dots \circ q_{1}$. It is enough to check that, for every choice of $r_1, \dots, r_m$ (not all
	zero) we have $\frac{\partial Q_m (\eps,w)}{\partial \eps}\neq 0$
	(as a polynomial in $w$) 
	at $\eps=0$. 
	Since for all $1\leq j\leq m$ we have $Q_j (0,w) =w^{d^{j}}$, we can check by induction that
	\[
	\begin{aligned}
\frac{	\partial Q_1 (\eps, w)	}{\partial \eps} \Big|_{\eps=0} &= r_1(w);\\
\vdots\\
\frac{	\partial Q_j (\eps, w)	}{\partial \eps} \Big|_{\eps=0} & 
= r_j (w^{d^{j-1}}) + d(w^{d^{j-1}\cdot (d-1)}) \cdot  \frac{\partial Q_{j-1} (\eps, w)	}{\partial \eps} \Big|_{\eps=0};\\
\vdots\\
\frac{	\partial Q_m (\eps, w)	}{\partial \eps} \Big|_{\eps=0} & = r_m (w^{d^{m-1}}) + d(w^{d^{m-1} \cdot (d-1)}) \cdot  \frac{\partial Q_{m-1} (\eps, w)	}{\partial \eps} \Big|_{\eps=0}.
\end{aligned}
	\]
Since $\deg r_j\leq d-2$ for all $j$, it follows that, for all $0\leq j\leq m-1$, $ \frac{\partial Q_{j+1} (\eps, w)	}{\partial \eps} \Big|_{\eps=0}\neq 0$
as soon as $ \frac{\partial Q_{j} (\eps, w)	}{\partial \eps} \Big|_{\eps=0}\neq 0$.
Hence, in order to have $ \frac{\partial Q_{m} (\eps, w)	}{\partial \eps} \Big|_{\eps=0}=0$, we must have
$ \frac{\partial Q_{j} (\eps, w)	}{\partial \eps} \Big|_{\eps=0}= 0$ for all $0\leq j \leq m$.
Since this implies that all the $r_j$'s must be equal to 0, the proof is complete.
	\end{proof}
	
	 Therefore, since $Q_{\lam,z}^m = C \circ \phi_z(\lam)$, 
	the map $\lam \mapsto Q_{\lam,z}^m$ is locally injective near $\lam:=0$, and 
	so the family $(Q_{\lam,z}^m)_{\lam \in \skpd}$ has indeed dimension $D_d$.
	
	Once this property is established,  
	the statement follows from a slight adaptation of the main result in \cite{gorbovickis2015}, 
	which is as follows.
	For any $D \geq 2$, for any $Q_{\lam_0} \in \pol(D)$, let $w_i$ ($1 \leq i \leq D-1$)
	be a collection of repelling periodic points for $Q_{\lam_0}$, of distinct periods $m_i$. 
	Up to passing to a finite branched cover of $\pol(D)$, we may follow globally those periodic points as functions of the parameter
	$\lam \mapsto w_i(\lam)$. If we denote by $\rho_i(\lam)$ their respective multipliers $\rho_i(\lam):=(Q_\lam^m)'(w_i(\lam))$
	and set $\rho(\lam):=(\rho_i(\lam))_{1\leq i \leq D-1}$, 
	then 
	Gorbovickis proves that
	there exists a global hypersurface $\mathcal{R}$ such that for all $\lam \notin \mathcal{R}$, 
	the differential $d\rho_\lam$ is invertible.
	
	Therefore, it is enough for us to arbitrarily complete our collection of repelling periodic
	points with some $w_i$ (with $D_d < i \leq d^m$) and prove
	 that the subfamily $(Q_{\lam,z}^m)_{\lam \in \skpd} \subset \pol(d^m)$
	is not entirely contained in the corresponding algebraic hypersurface $\mathcal{H} \subset \pol(d^m)$.
	But this in turn follows from the facts that 
	 $w \mapsto w^{D}$ never belongs to 
	$\hcal$ (\cite[Lemma 2.1]{gorbovickis2015}), 
	and 
	that $w \mapsto w^{d^m}$ always belongs to
	$(Q_{\lam,z}^m)_{\lam \in \skpd} \subset \pol(d^m)$.
	The proof is complete.
\end{proof}

\begin{lem}\label{lem:simplecrit}
	Let $\lam_0 \in \skpd$ and assume that $f_{\lam_0}$ has a Misiurewicz relation $f_{\lam_0}^n(z_0,c_0)=(z_1,w_1)$
	satisfying \ref{item_good_znotcrit},
	and let $m$ be the period of $(z_1,w_1)$.
	Let $M_{(z_0,c_0),n} \subset \skpd$ denote the local hypersurface of $\skpd$ where this Misiurewicz relation is preserved. 
	Then,
	the set of parameters $\lam \in M_{(z_0,c_0),n}$ which satisfy  \ref{item_good_simple} and \ref{item_good_nontangent} is open and dense in $M_{(z_0,c_0),n}$.
\end{lem}

\begin{proof}
 	It will be useful to 
	consider the \emph{algebraic} hypersurface $\mcal \subset \skpd$ defined by:
	$$\mcal:=\{\lam \in \skpd: \res(q_{\lam,z_0}', Q_{\lam,z_0}^{n+m}-Q_{\lam,z_0}^{m} )   \}$$
	where $\res$ is the resultant and $m$ is the period of $(z_1, w_1)$. In other words, $\mcal$ is the set of $\lam \in \skpd$ such
	 that a critical point
	 of the form $(z_0,c)$ 
	 lands after $n$ iterations on a periodic point of period dividing $m$
	(and that periodic point may or may not be repelling).
	By definition, $M_{(z_0,c_0),n}$ is a neighbourhood of $\lam_0$ in $\mcal$.

	Let us first prove that the subset of $\mcal$ where  \ref{item_good_simple} does not hold has codimension at least 1 in $\mcal$.
	Let $\Pi:=\{ (p,q): q \in \pol(d)\}\subset \skpd$
	 denote the subfamily of trivial products. 
	Then \ref{item_good_simple} holds
	on a dense open subset of $\mcal \cap P$, and so \ref{item_good_simple} does not hold
	 on a subset of $\mcal$ of codimension 
	at least $1$ (in fact exactly 1, unless $d=2$ in which case \ref{item_good_simple} is always true).

	Now let $\lam_1 \in \mcal$ be a parameter where \ref{item_good_simple} holds. Then we may locally follow the critical point $(z_0,c_0)$ 
		as $(z_0, c_0(\lam))$, and moreover there exists a unique 
		irreducible component $L_\lam$ of $\crit(f_\lam)$ passing through $(z_0, c_0(\lam))$ for $\lam$ close enough to $\lam_0$, of local equation of the form $w=c(\lam,z)$. By \ref{item_good_znotcrit}, the algebraic set 
		$f_{\lam}^n(L_\lam)$ is also locally a graph near $(z_1,w_1(\lam))$, with local equation 
		given by $w=Q_{\lam,p^{-n}(z)}^n(c(\lam,z))$, where $p^{-n}$ denotes the local inverse branch of $p^n$ mapping $z_1$ to $z_0$. In particular, it is regular at $(z_1, w_1(\lam))$ and its tangent space is not vertical. 
		
		We now need to prove that the set of $\lam_1 \in \mcal$ where the tangent space of $f_\lam^n(L_\lam)$ is not an eigenspace of $(df_\lam^m)_{(z_1, w_1(\lam))}$ is open and dense in $\mcal$; again by the algebraicity of the condition,
		it is in fact enough to prove that this subset is non-empty.
Hence, we can restrict ourselves to the unicritical subfamily 
$U_d\subseteq \skpd$
 introduced in Section \ref{ss_unicritical} and prove the analogous statement there.
By Lemmas \ref{lemma_dir_inv} and \ref{lemma_dir_pc},  and with the notations
as in those lemmas, 
it is enough to prove that 
 the identity $v^{(2)}_\lam=u^{(2)}_\lam$ 
 cannot 
 hold on all of $M\cap U_d$.
 
By Lemma \ref{lemma_ez}, the accumulation on $\P^d_\infty$
of $M\cap U_d$ is equal to $E_{z_0} = \{[\lam]\colon a(z_0)=0\}$.
We choose $\lam_\infty$ such that  $[\lam_\infty] \in E_{z_0}$ and
$z_0$ is the only 
root in $J_p$ of the derivative $a'$ of the associated polynomial $a$.
Since
 there exists a sequence $(\lam_j)_{j \in \N}$ such that for all $j \in \N$, $\lam_j \in M\cap U_d$ and 
$[\lam_j] \to [\lam_\infty]$,
the assertion  follows from
 the estimates \eqref{eq_dir_inv} and \eqref{eq_dir_pc}. The proof is complete.
\end{proof}

We can now prove Theorem
\ref{teo_main_no_current} and the other results stated in the Introduction.

\begin{proof}[Proof of Theorem \ref{teo_main_no_current}]

		Fix $\lambda_0 \in \Bif(\skpd)$ and $\eps>0$. Set $M^0:=\skpd$,
		and let $H_{\lam_0}$ denote the limit set of a vertical-like
		IFS for $f_{\lam_0}$
		(which exists by Proposition \ref{prop:modrep}).
		Let
		$z$ be a repelling periodic point
		 of period $m>d$
		 for $p$ and
		 $(z, w_i(\lam_0))$ be a collection of repelling periodic points in $H_{\lam_0}$ as in 
		Lemma \ref{lem:multtrans}.

		  We will prove by induction on $1 \leq k \leq \dim \skpd$ 
		that there exist
	a parameter $\lam_k$	which is $k\eps$-close to $\lam_0$ and
		 a family $M^k$ with $\lam_k\in M_k$
		satisfying the following properties:
		\begin{enumerate}[label={\bf (I\arabic*)}]
			\item\label{item_thm_induction_intersection}  $M^k=\bigcap_{1 \leq i \leq k} M_{(y_i,c_i),(z, w_i(\lam_k)), n_i}$
			is the intersection of $k$ distinct Misiurewicz loci (where a critical point lands after some
			iterations on one of the periodic points $(z, w_i)$ introduced above);
			\item\label{item_thm_induction_dimension} $M^k$ has codimension $k$ in $\skpd$;
			\item\label{item_thm_induction_good} if $k<\dim \skpd$, among the $k$ persistent
			Misiurewicz relations defining $M^k$, at least one is good in the sense 
			of Definition \ref{def:goodmis} in a neighbourhood of $\lam_k$.
		\end{enumerate} 

	Recall that each $M_{(y_i,c_i),(z, w_i(\lam_k)), n_i}$ is a local family; in particular, condition \ref{item_thm_induction_good} is also local.

		\subsection*{Initialization: the case $k=1$}
		
		Using the notation of Lemma \ref{lem:multtrans}, up to replacing $\lam_0$
		by a first perturbation $\lam_0'$, we may assume without
		loss of generality that $\lam_0 \in \Bif(\skpd) \backslash R$. Indeed, the
		bifurcation locus cannot be locally contained in 
		any proper analytic subset of $\skpd$, since
		the bifurcation current has continuous potential. 
		In the rest of the proof, we will always assume that all perturbations are small enough
		so that none of the parameters we consider belong to $R$.

		We then apply Lemma \ref{lem:landingonrep} to find $\lam' \in \mathbb{B}(\lam_0,\frac{\eps}{3})$ such that 
		$f_{\lam'}$ has a Misiurewicz relation of the
		form
		$f^{n_1}_{\lam} (y_1, c_1) = (z,w_1 (\lam))$, hence 
				$\lam' \in  M_{(y_1,c_1),(z, w_1(\lam')), n_1}$.
Here $n_1$ can be taken arbitrarily large.
        We can assume that  $y_1$ satisfies $(p^n)' (y_1)\neq 0$,
          that it does not belong to the cycle of $z$ and that the period $m_1$ 
          of $z$ satisfies $m_1>d$.
		We need to prove that up to perturbing $\lam'$ inside  $M_{(y_1,c_1),(z, w_1(\lam')), n_1}$,
		we can obtain  $\lam_1 \in M_{(y_1,c_1),(z, w_1(\lam')), n_1} \cap \B(\lam_0,\eps) $ which is a good parameter in
		the sense of Definition \ref{def:goodmis}.
		
		Let us first prove that the vertical multiplier $\rho_1(\lam)$ 
		of $(z,w_1(\lam))$
		is not constant on $M_{(y_1,c_1),(z, w_1(\lam')), n_1}$. The argument is similar to the one in the proof of Lemma \ref{lem:simplecrit}: we consider 
		the intersection $\tilde M := M_{(y_1,c_1),(z, w_1(\lam')), n_1} \cap U_d$
		with the unicritical subfamily $U_d$
		 and pick $[\lam_\infty]$
		is the accumulation on $\P^d_\infty$  of
$\tilde M$
		 such that $\lam_\infty(p^i(z)) \neq 0$ for all $0 \leq i \leq m_1$.
		 Then, by Lemma \ref{lemma_ez}, there exists a
		 sequence $(\lam_j)_{j \in \N}$ such that 
  $\lam_j \in \tilde M$ for all
		 $j \in \N$, 
		 $|\lam_j| \to +\infty$ and $[\lam_j] \to [\lam_\infty]$.	
		By Lemma \ref{lemma_claim_12}, for all $0 \leq i \leq m_1-1$, we have $|w_{i,j}| \asymp |\lam_j|^{1/d}$ where 
		$w_{i,j}:=Q_{\lam_j,z}^{n_1+i}(0)$. In particular, 
	$|\rho_1(\lam_j)|:=\left|(Q_{\lam_j,z}^{m_1})'(w_{0,j}) \right| \asymp |\lam_j|^{m_1 \cdot (d-1)/d}$
		and therefore is not constant. This proves \ref{item_good_non_constant}.

	Observe that \ref{item_good_znotcrit} follows from the choice of $y_1$
		 as in Lemma \ref{lem:landingonrep}.
		Property \ref{item_good_vertical_like} is a consequence of the fact that $(z,w_1(\lam))$ belongs to the limit
		 set of the vertical-like IFS, and it follows from Lemma \ref{lem:simplecrit} that  properties
		  \ref{item_good_simple} and \ref{item_good_nontangent} are generically satisfied in 
		 $M_{(y_1,c_1),(z, w_1(\lam')), n_1}$ 
		 as soon as \ref{item_good_znotcrit} holds. 
		 This takes care of \ref{item_thm_induction_intersection} and \ref{item_thm_induction_good};
		 and \ref{item_thm_induction_dimension} is obvious in the case $k=1$.

		\subsection*{Heredity}
		Assume now that $k<D_d$ is such that there exists $\lam_k$ satisfying \ref{item_thm_induction_intersection}, \ref{item_thm_induction_dimension}
		and \ref{item_thm_induction_good}.
		By the induction hypothesis, there exists $k_0 \in \{1, \ldots, k\}$ such that the vertical multiplier $\rho_{k_0}$ of the
			repelling cycle from the $k_0$-th Misiurewicz relation $M_{(y_{k_0},c_{k_0}),(z, w_{k_0}(\lam))}$
		is non-constant on $M^k$. Consider the germ of analytic subset of $M^k$ defined by
		$N:=\{\lam \in M^k : \rho_{k_0}(\lam)=\rho_{k_0}(\lam_k) \}$. Then $N$ has codimension $k+1$ in $\skpd$.
		We claim that if $k<D_d-1$, there exists at least one  repelling point $(z,w_{j_0})$ 
		(among all those introduced at the beginning of the proof) 
		with $j_0 \neq k_0$ such that its vertical multiplier $\rho_{j_0}$ is non-constant on $N$.
		 Indeed, by Lemma \ref{lem:multtrans}, $\dim \bigcap_{i=1}^{D_d} \{\lam \in \skpd : \rho_i(\lam)=\rho_i(\lam_k) \} = 0$, while if 
		$k < D_d-1$ then $\dim N>0$. If $j_0>k$, then we relabel the repelling periodic points $(z,w_i)_{k+1\leq i \leq D_d}$ 
		so that $j_0=k+1$.
		
		We now take $\lam_{k+1,\infty} \in \B(\lam_k,\frac{\eps}{2})$ to be a point 
		 in the dense set $S$ given by
		 Proposition \ref{prop:secmis}, and then take $\lam_{k+1} \in \B(\lam_{k+1,\infty}, \frac{\eps}{2})$
		 such that $\lam_{k+1}$ has a Misiurewicz relation as in Proposition \ref{prop:secmis}, with 
		$(z'',w''):=(z, w_{k+1})$.
We consider the associated Misiurewicz locus  $M^k \cap  M_{(y_{k+1},c_{k+1}),(z, w_{k+1}(\lam_{k+1})), n_{k+1}}$.
	By Proposition \ref{prop:secmis}, $\lam_{k+1}$ already satisfies \ref{item_thm_induction_intersection}
	 and \ref{item_thm_induction_dimension} with
		$$M^{k+1}:= M^k \cap  M_{(y_{k+1},c_{k+1}),(z, w_{k+1}(\lam_{k+1})), n_{k+1}}.$$
		If $k=D_d-1$, we are done; otherwise, it remains to be proved that at least 
		one of the Misiurewicz relations defining $M^{k+1}$ is good in $M^k$.
		
	Note that items \ref{item_good_vertical_like}, 
		 \ref{item_good_znotcrit},
			\ref{item_good_simple}, and \ref{item_good_nontangent} 
			are all preserved by restriction, so that they still hold 
			 on $M^{k+1}$ for each of the first 
		$k$ Misiurewicz relations $ M_{(y_i,c_i),(z, w_i(\lam_k)), n_i}$ (with $1 \leq i \leq k$).
		Moreover, the new Misiurewicz relation  $ M_{(y_{k+1},c_{k+1}),(z, w_{k+1}(\lam_{k+1})), n_{k+1}}$
		satisfies \ref{item_good_vertical_like} since $(z,w_{k+1})$
	  is vertical-like by definition, and satisfies
 \ref{item_good_znotcrit},
  \ref{item_good_simple}, and \ref{item_good_nontangent} by Proposition \ref{prop:secmis}.
		
		It now only remains to prove that at least one among 
		the $(z, w_i(\lam))$ (for $1 \leq i \leq k+1$) has a non-constant vertical multiplier on $M^{k+1}$, which would give
		\ref{item_good_non_constant}. 
		Recall that there exists $k_0, j_0 \leq k+1$ with $k_0 \neq j_0$, such that 
		$\rho_{j_0}$ is non-constant on $\{\lam \in \skpd: \rho_{k_0}(\lam) = \rho_{k_0}(\lam_k)  \}$.
		In other words, the level sets $\{\lam \in M^k : \rho_{k_0}(\lam) = \rho_{k_0}(\lam_k)\}$ and 
		 $\{\lam \in M^k : \rho_{j_0}(\lam) = \rho_{j_0}(\lam_k)\}$ are two distincts analytic hypersurfaces of $M^k$. 
		Up to taking $\lam_{k+1}$ close enough to $\lam_k$, we may still assume that 
		the same holds at $\lam_{k+1}$. Therefore $M^{k+1}$
		(which has codimension 1 in $M^k$) cannot be  contained in  
		 \[\{\lam \in M^k : \rho_{k_0}(\lam) = \rho_{k_0}(\lam_{k+1})\} \cap \{\lam \in M^k : \rho_{j_0}(\lam) = \rho_{j_0}(\lam_{k+1})\},\]
		 which precisely means that
	either $\rho_{j_0}$ or $\rho_{k_0}$ is non-constant on $M^{k+1}$.	
		
		Therefore, at least one of the $k+1$ Misiurewicz relations defining $M^{k+1}$ is good in the sense of definition 
		\ref{def:goodmis}. This proves \ref{item_thm_induction_good} and completes the inductive step. The proof is complete.
		\end{proof}

		\begin{proof}[Proof of Theorem \ref{teo_main}]
Let $\lam_{D_d}$ be as constructed in the proof of Theorem \ref{teo_main_no_current}.		
		By Proposition \ref{prop_sufficient_condition_bifk}, 
		we have $\lam_{D_d} \in \Supp \tbif^{D_d}$. This gives $\Supp \tbif = \tbif^{D_d}$, and proves the assertion.
\end{proof}

\begin{proof}[Proof of Corollary \ref{big_families_mubif_all}]
By the initialization step in the proof of Theorem \ref{teo_main_no_current}, for every $d$
there exists a  Misiurewicz hypersurface of $\skpd$
 which is good in the sense of Definition \ref{def:goodmis}. The result follows from Corollary \ref{cor:bifm=m}.
\end{proof}

\begin{proof}[Proof of Corollary \ref{cor_open}]
By \cite{dujardin2016non,taflin_blender}, for every $d\geq 2$ the bifurcation locus of the family
$\skpd$ is not empty.
The assertion follows from Theorem \ref{teo_main}.
\end{proof}

\begin{proof}[Proof of Corollary \ref{cor_intro_hausdorff}]
By Theorem \ref{teo_main} it is enough to check that the same property is true for the bifurcation locus. 
By  \cite[Theorem 3.3]{astorg2018bifurcations}, the bifurcation locus associated to the return maps of any
periodic fibre is  
contained in the bifurcation locus of the family $\skpd$.
 By  \cite{mcmullen2000mandelbrot} the bifurcation loci of the return maps
  have full Hausdorff dimensions. The assertion follows.
\end{proof}

\appendix

\section{Proof of Lemma \ref{lemma-BD-fibered}}\label{as:proof-lemma-fibered}\label{a:proof_lemma_tf}

We work here in the assumptions of Proposition \ref{prop:modrep}.
We assume that we are given a
hyperbolic set 
$\tilde H$ in $J_p$ as in the proof of Proposition \ref{prop:modrep}, i.e., with positive entropy and
$\delta:=\dim_H \tilde H >1$. We will be only interested in the following in the dynamics
of $f$ on $\tilde H$ (and $\tilde H \times \C$).
We denote by $\tilde m$ the \emph{conformal measure} associated
with the weight $\tilde \phi (z) :=-\log |p'(z)|^\delta$. Recall
that this means that $\tilde m$ is an eigenvector for the dual $\tilde \Ll^*$ of the Perron-Frobenius
operator $\tilde \Ll$
acting on continuous functions  $g\colon \tilde H \to \R$ as 
\[\tilde \Ll_{\tilde \phi} (g)(x) =\sum_{f (a)=x}  e^{\tilde \phi (a)} g (a). \]
Observe that 
$\tilde \phi$
is H\"older continuous on $\tilde H$. 
This and the hyperbolicity of $\tilde H$ imply the existence and uniqueness of $\tilde m$, see for instance
\cite{przytycki2010conformal}.
Moreover, $\tilde m$
is equivalent to the 
$\delta$-dimensional
Hausdorff measure. 
It is also equivalent to the unique \emph{equilibrium state} $\tilde \nu$
for the system $(\tilde H,f)$
associated with $\tilde \phi$.  This means
that $\tilde \nu$
 is the (unique)
maximizer of the \emph{pressure}
$P(\tilde \phi) := \sup_{\omega} h_{\omega}+ \langle\omega, \tilde \phi\rangle$,
where the supremum is taken over all invariant probability measures for $f$ and
$h_\omega$ is the metric entropy of the measure $\omega$.
We denote by $\rho$
the Radon-Nikodym derivative of $\tilde \nu$ with respect to $\tilde m$, i.e., set
$\tilde \nu = \rho\tilde m$, and 
by $L_{\tilde \nu}= \langle \tilde \nu, \phi\rangle$ 
the Lyapunov exponent for $\tilde\nu$.
By the construction of $\tilde H$, 
we have $L_{\tilde \nu} < \log d  \leq  L_v$, where $L_v$ is the vertical exponent for $f$.
We also have
\begin{equation}\label{eq-equid-E}
\lim_{n\to \infty}\lam^{-n} \sum_{p^n (a)=x}\frac{1}{ |(p^n)' (a)|^{\delta}} g(a) \to \rho(x)\langle m, g\rangle 
\end{equation}
for all  $x \in \tilde H$  and continuous functions $g\colon \tilde H\to \R$, where $\lam$
is the eigenvalue corresponding to $\tilde m$, i.e., $\tilde \Ll^* \tilde m = \lam \tilde m$.

Recall that $J_{\tilde H} 
= \bar {\cup_{z \in \tilde H} \{z\}\times J_z}$. 
We see $(J_{\tilde H},f)$
 as a dynamical system and
we can consider the weight on $J_{\tilde H}$ given by
$\phi(z,w)= \tilde \phi (z) =-\log |p' (z)|^\delta$.
 Observe that, a priori, $\phi$
 is not a H\"older continuous weight on $J$, and the system 
 $(J_{\tilde H},f)$ is not necessarily hyperbolic.
 Hence, we cannot directly apply the thermodynamical
formalism 
for the system $(J, f)$  and weight $\phi$
(see for instance \cite{urbanski2013equilibrium, bd-eq-states}), 
nor to the system $(J_{\tilde H},f)$.
However,
given the fibred structure
it is immediate to deduce the following result.

\begin{lemma}\label{lemma_nu_fibered}
The measures
\[
 m := \int_{\tilde H} \mu_z d\tilde m(z) 
\quad \mbox{ and }\quad
 \nu := \int_{\tilde H} \mu_z d\tilde \nu(z) = \int_{\tilde H} \mu_z \rho(z) d\tilde m(z)
\]
are
the unique conformal measure and equilibrium state associated with the weight 
$\phi(z,w)= -\log |p' (z)|^\delta$ on the system $(J_{\tilde H}, f)$,  respectively.
The measure $\nu$
is invariant, mixing, and satisfies
\begin{equation}\label{eq-equid-JE}
\lim_{n\to \infty} (d\lam)^{-n} \sum_{f^n (a)=x}\frac{1}{ |(p^n)'(\pi_z(a))|^{\delta}} g(a) \to \rho(x)\langle m, g\rangle 
\end{equation}
for all  $x \in J_{\tilde H}$  and continuous functions $g\colon J_{\tilde H}\to \R$. Moreover, the metric entropy
of $\nu$ is strictly larger than $\log d$
and the Lyapunov exponents of $\nu$ are equal to $L_v$ and $L_{\tilde \nu}$. In particular, they are strictly positive.
\end{lemma}

\begin{proof}
It follows from the definition that $\nu$ is $f$-invariant. Moreover, \eqref{eq-equid-E} implies \eqref{eq-equid-JE}
and \eqref{eq-equid-JE} implies that $m$ is a (unique) conformal measure and that $\nu$ is mixing
and is a unique equilibrium state, see for instance \cite{przytycki2010conformal,urbanski2013equilibrium,bd-eq-states}.
The metric entropy of $\nu$ is equal to $\log d + h_{\tilde \nu}$ 
by Brin-Katok formula  \cite{brin1983local}
for the mass of infinitesimal balls and the fibred structure  of $\nu$.
Both $L_{\tilde \nu}$ and $L_v$ must be Lyapunov exponents for the systems, which completes the
proof.
\end{proof}

In order to prove Lemma \ref{lemma-BD-fibered},
we will give an adapted fibered version of the proof by Briend-Duval \cite{briend1999exposants}
 of the equidistribution of
repelling
periodic points with respect to the equilibrium measure
for endomorphisms of $\P^k$ 
(which is done by constructing enough contracting inverse branches of a ball
centred on the Julia set to itself).
Since the method is now standard, we just sketch the overall proof and
highlight
the differences here
(due to the non-constant Jacobian of $\nu$ and $\tilde \nu$).
More details can be found in 
\cite[Section 4.7]{bd-eq-states}, where the strategy is adapted to prove
the equidistribution of repelling periodic points with respect to the equilibrium state 
(when the weight satisfies some regularity
condition on all of $J$).

First we need to introduce the natural extension of the system $(J_{\tilde H},f)$, see \cite{cornfeld2012ergodic}. We set
$X:= J_{\tilde H} \setminus \cup_{m\geq 0} f^{-m} (PC_f)$, where $PC_f := \cup_{n\geq 0} f^{n}C_f$
is the postcritical set of $f$. Since, by Lemma \ref{lemma_nu_fibered}, the entropy of $\nu$ is strictly
larger than $\log d$, we have $\nu (X)=1$, see \cite{de2008exposants,dupont2012large}. We then consider the
dynamical system $(\hat X, \hat f, \hat \nu)$, where
$\hat X =\{ \hat x=(\dots x_{-1}, x_0, x_1, \dots ) \colon f(x_i)=x_{i+1}\}$
and $\hat f (\hat x)= (x_{i+1})_{i\in \Z}$ where $\hat x =(x_i)_{i\in \Z}$. The measure
$\nu$ lifts to a measure $\hat \nu$ satisfying $(\pi_0)_* \hat \nu = \nu$, where $\pi_0 \colon \hat X\to X$
is given by $(x_i)\mapsto x_0$. The measure $\hat \nu$ is mixing since $\nu$ is mixing.
For any $\hat x\in \hat X$ and every $n$ we denote by $f^{-n}_{\hat x}$
the inverse branch of $f^n$ in a neighbourhood of $x_0$ with values in a neighbourhood of $x_{-n}$.
We have the following lemma.

\begin{lemma}\label{lemma-inv-branches}
For all $\eps < L_{\tilde \nu}$
there exist
 measurable functions $r_\eps, L_\eps, T_\eps\colon \hat X \to \R^+$
 such that, for $\hat \nu$-almost all
  $\hat x \in \hat X$ and  all $n\geq 1$,
  \begin{enumerate}
  \item\label{item_ib_def} the map $f^{-n}_{\hat x}$ is defined on $B(x_0, r_\eps (\hat x))$;
\item\label{item_ib_contraz} $\Lip (f^{-n}_{\hat x} )\leq L_\eps e^{-nL_{\tilde \nu} + n\eps}$ on $B(x_0, r_\eps (\hat x))$;
\item\label{item_ib_sum} $\forall y \in f^{-n}_{\hat x} (B(x_0, r_\eps (\hat x)) )$ we have $\abs{\frac{1}{n} \log |\jac\, df_y^n| - (L_{\tilde \nu} +L_v) }\leq \frac{1}{n} \log T_\eps (\hat x) + \eps$;
\item\label{item_ib_max} $\forall y \in f^{-n}_{\hat x} (B(x_0, r_\eps (\hat x)) )$ we have $\abs{\frac{1}{n} \log \|df_y^n\| - L_v}\leq \frac{1}{n} \log T_\eps (\hat x) + \eps$.
  \end{enumerate}
\end{lemma}

\begin{proof}
The statement is a consequence of \cite[Theorem 1.4]{berteloot2008normalization}, see also
\cite[Theorem A]{berteloot2019distortion}. These results are stated for the measure of maximal entropy, but 
only the strict positivity of the Lyapunov exponents of the measure is needed, see the remark
at the end of the Introduction of \cite{berteloot2019distortion}.
\end{proof}

We fix $\eps\ll L_{\tilde \nu}$ in what follows and
set $\hat X_{C}:= \{\hat x \in \hat X \colon r_\eps^{-1}, L_\eps, T_\eps < C\}$.
We have $\hat \nu(\hat X_C) \to 1$ as $C\to \infty$.
Given a Borel subset $E \subset \C^2$, we set
\[
\hat E := \pi_0^{-1} (E \cap X),
 \quad
 \hat E_C := \hat E \cap X_C, 
 \quad
 \mbox{ and }
 \quad
 \nu_C = (\pi_0)_* (\hat \nu_{| \hat X_C}).
\]

Fix now a point $x \in X$, 
a constant $C$ sufficiently large (to be chosen later),
and a fibred box $x \in A \subset B(x, 1/(2C))$. 
We also fix 
 a
subset 
 $A_r:= \{y \in A,  \dist (y, A^c)>r\}$,
where the complement $A^c$ 
is taken in $J_{\tilde H}$.

We call \emph{good component of $f^{-n} (A)$}
any connected 
component with diameter smaller than
$r/2$. Since any  good component intersecting
$A_r$ is strictly  contained in $A$, to prove the lemma we 
need to show that (we can choose $A, C,r$ so that) 
for  $n$ sufficiently large,
there are at least $3d^n$ good components of $f^{-n} (A)$
intersecting $A_r$ and satisfying the estimates in \eqref{eq_est_cone}.

Notice that, for any $y \in \hat A_C$
the inverse branch $f_{\hat y}^{-n}$
is defined on $A$. 
Moreover, 
it follows from Lemma \ref{lemma-inv-branches}\eqref{item_ib_contraz}
that,
for all $n$ sufficiently large
all images of such inverse branches have diameter smaller than $r/2$
(uniformly in $\hat y$). Hence they are good components.

Since
$\hat \nu$
is mixing, 
we have $\hat \nu (\hat f^{-n} (E_1) \cap E_2)\to  \nu (E_1) \cdot \nu (E_2)$
for any Borel subsets $E_1, E_2\subseteq \hat X$. In particular, we
have, for all $n$ large enough,
\[
\nu ( \pi_0 (\hat f^{-n} (\hat A_r)_C  ) \cap A_r) 
 =
\hat \nu (\hat f^{-n} (\hat A_r)_C \cap \hat A_r)
 \geq  \frac{1}{2}  \hat \nu ((\hat A_r)_C) \cdot \hat \nu (\hat A_r) 
 = \frac{1}{2}   \nu_C ( A_r ) \cdot  \nu ( A_r).
\]

By the argument above, the LHS of the above expression is larger that $\nu (\cup_j A^j)$ where $A^j, 1\leq j \leq N$,
are the good components of $f^{-n} (A)$ intersecting $A_r$. To get the desired estimate on $N$, we need to find 
an upper bound for $\nu (A^j)$
(this bound is immediate when working with the measure of maximal entropy, since
this measure has constant Jacobian).
 We use here the definition of fibred box, letting $a$ be
the common $\mu_z$-measure of
all the non empty slices $A\cap ({z}\times \C)$.
We then have  $\nu(A)= a\tilde\nu (B)$
(where $B$ is the projection of $A$ on the first coordinate) and so $\nu (A^j)= a \tilde\nu (B^j) /d^n$, where
$B^j$ is the projection of $A^j$ on the first coordinate.

Since the measure $\tilde \nu$ is not-atomic, the function $M(r):= \sup_{z \in \C} \tilde \nu (B (z,r))$
goes to $0$ as $r\to 0$. Since the system
$(\tilde H,f)$ is hyperbolic,
the diameters of all the $B^j$ tend uniformly to zero
as $n\to \infty$. Hence, there exists a function $M'_n$ such that $\tilde \nu(B^j)\leq M'_n$ for all $j$ and
$M'_n \to 0$ as $n\to \infty$.
Take
$n$ large enough so that $M'_n < 1/6$. The above inequalities imply that $N> 3d^n$, as desired.
The estimates in \eqref{eq_est_cone} follow from items 
\eqref{item_ib_sum} and \eqref{item_ib_max} in Lemma \ref{lemma-inv-branches} (up to possibly increasing $n$).

\bibliographystyle{alpha}
\bibliography{biblio}

\end{document}